\documentclass[11pt]{amsart}
\usepackage[foot]{amsaddr}
\usepackage{amssymb,amsthm,amsmath,amstext}
\usepackage{mathrsfs}  
\usepackage{mathtools}
\usepackage{mathdots}
\usepackage{xcolor}
\usepackage{multirow}
\usepackage{enumitem}
\usepackage{colonequals}
\usepackage{cases}
\usepackage{graphicx}
\usepackage{bm}

\usepackage{tikz}
\tikzset{every picture/.style={line width=0.75pt}} 
\usepackage{physics}
\usepackage{amsmath}
\usepackage{tikz}
\usepackage{mathdots}
\usepackage{yhmath}
\usepackage{cancel}
\usepackage{color}
\usepackage{siunitx}
\usepackage{array}
\usepackage{multirow}
\usepackage{amssymb}
\usepackage{gensymb}
\usepackage{tabularx}
\usepackage{extarrows}
\usepackage{booktabs}
\usetikzlibrary{fadings}
\usetikzlibrary{patterns}
\usetikzlibrary{shadows.blur}
\usetikzlibrary{shapes}

\usepackage{subfig}

\usepackage[left=1.0in,top=0.95in,right=1.0in,bottom=0.95in]{geometry}

\usepackage[all]{xy}

\newcommand{\PP}{\mathbb{P}}

\newcommand{\M}{\mathcal{M}}

\newcommand{\g}[2]{g^{#1}_{#2}}
\newcommand{\BN}[3]{\mathcal{M}^{#2}_{#1,#3}}

\newcommand{\floor}[1]{\left\lfloor #1 \right\rfloor}
\newcommand{\ceil}[1]{\left\lceil #1 \right\rceil}

\newcommand{\isom}{\cong}

\usepackage[backref=page]{hyperref}
\hypersetup{pdftitle={Some reducible and irreducible Brill-Noether loci},pdfauthor={Richard Haburcak and Montserrat Teixidor i Bigas}} 
\hypersetup{colorlinks=true,linkcolor=blue,anchorcolor=blue,citecolor=blue}

\usepackage[nameinlink]{cleveref}

\newtheorem{theorem}{Theorem}[section]
\newtheorem{lemma}[theorem]{Lemma}
\newtheorem{prop}[theorem]{Proposition}

\newtheorem{cor}[theorem]{Corollary}

\theoremstyle{definition}
\newtheorem{defn}[theorem]{Definition}

\theoremstyle{definition}
\newtheorem{remark}[theorem]{Remark}
\newtheorem{example}[theorem]{Example}

\newtheorem{theoremintro}{Theorem}

\title{Some reducible and irreducible Brill--Noether loci.}

\author{Richard Haburcak}
\address{Department of Mathematics\\%
	The Ohio State University\\%
	231 W 18th Ave\\%
	Columbus, OH 43210\\%
	USA\\%
	ORCiD 0000-0002-6591-8918}
\email{haburcak.1@osu.edu}

\author{Montserrat Teixidor i Bigas}
\address{Mathematics Department, Tufts University\\%
	177 College Ave\\%
	Medford, MA 02155\\%
	USA\\%
	ORCiD 0000-0003-3747-3330}
	\email{mteixido@tufts.edu}

\title{Some reducible and irreducible Brill--Noether loci}
\date{}

\begin{document}

\thispagestyle{empty}
\vspace*{-.9cm}

\begin{abstract} We investigate limit linear series on chains of elliptic curves, giving a simple proof of a conjecture of Farkas stating the existence of curves with a theta-characteristic with a given number of sections for the expected range of genera. Using the additional structure afforded by considering limit linear series on chains of elliptic curves, we find examples of reducible Brill--Noether loci, admitting at least two components, with and without a theta-characteristic respectively. This allows us to display reducible Hilbert schemes for $r\ge 3$ and $d=g-1$. We also give examples of Brill--Noether loci with three components. On the positive side, we provide optimal bounds on the degree under which Brill--Noether loci are irreducible when $r=2$.
\end{abstract}

\maketitle

\vspace*{-0.9cm}

\section*{Introduction}\label{Introduction}

The set of linear series of degree $d$ and dimension $r$ on a generic curve of genus $g$ forms a scheme of expected dimension $\rho(g,r,d)=g-(r+1)(g-d+r)$.
When this number is greater than or equal to 0, the space of such linear series is non-empty on any smooth projective curve of genus $g$ and is irreducible  of dimension precisely $\rho$ on the generic curve.
This gives rise to a component of the Hilbert scheme of smooth curves of degree $d$ in ${\mathbb P}^r$ for $r\geq 3$ of dimension $3g-3+\rho(g, r,d)+\dim \operatorname{Aut}({\mathbb P}^r)$.
One could naively expect that the Hilbert scheme actually coincides with this expected component and that it is well behaved.
 This philosophy may have prompted Severi to claim that the Hilbert scheme of curves in  ${\mathbb P}^r$ is irreducible when $g-d+r<0$. 
 The irreducibility was proved by Harris in \cite{HHilbert} for $r=3, d>\frac 54g+1$, by Ein for $r=3$, $d>g+3$, for $r=4$, $d\ge g+4$ 
or more generally for $d>\frac{2r-2}{r+2}g+\frac{2r+3}{r+2}$ in \cite{EinHilbP^3,EinHilbP^r} 
and by Iliev in \cite{I} for $r=5$, $d\ge \max(\frac{11}{10}g+2,g+5)$, along with many other cases, see for example \cite{Balico_fontanari}.
In the case $r=3$ the bound on the degree has been improved so that the Hilbert scheme of curves of genus $g$ and degree also $g$ has been proved to be irreducible when non-empty  (see \cite{KKi})
and similar results are known for $g+1, g+2$ (see \cite{KK}).

 On the other hand, Mumford (see \cite{MMurphyL}) was the first to provide a non-reduced component of the Hilbert scheme.
 Ein gave examples in \cite{EinHilbP^r} of additional non-reduced components of the Hilbert scheme for $r\ge 6$.
  Later Vakil showed in \cite{VMurphyL} that Hilbert schemes succumb to Murphy's law, 
  meaning  that it is possible to construct a smooth curve in projective space whose deformation space has any given number of components,
  with any given singularity type and even  non-reduced behavior. 
   A very down to earth example of a second component of the Hilbert scheme is provided in \cite{LVV}.
    In \cite{CIK1,CIK2,CIK3}, the authors obtain additional components of the Hilbert scheme by considering coverings of curves of positive genus.

  Most known examples of additional components of the Hilbert scheme live in the range where the Brill--Noether number  is positive 
 and therefore there is a single component that dominates   ${\mathcal M}_g$,  the moduli space of curves of genus $g$.
 Also, most examples correspond to fairly large values of $g$ and, as mentioned, $r\geq 6$ (and often a lot higher).
 
In this paper, we focus on the negative range of $\rho$, where the generic curve of genus $g$ does not possess a $g^r_d$, and those that do constitute a proper subset ${\mathcal M}^r_{g,d}$ of ${\mathcal M}_{g}$, a Brill--Noether locus, which has expected dimension $3g-3+\rho(g,r,d)$. The existence of components of expected dimension was recently shown when $\rho$ is not too negative by Pflueger in \cite{pflueger_lego} (see also \cite{DBN}).

The reducibility of Brill--Noether loci is an interesting question in itself of which little is known. 
When $\rho=-1$,  Eisenbud and Harris showed in \cite{EH} that ${\mathcal M}^r_{g,d}$ is irreducible. 
In the proof of \cite[Corollary~3.7]{CHOI2022}, Choi and Kim show that $\BN{g}{2}{d}$ is irreducible when $\rho(g,2,d)=-2$. 
Here, we show that if $\BN{g}{2}{d}$ has no components where the $g^2_d$ always has a basepoint, then $\BN{g}{2}{d}$ is irreducible if $d>\frac{g+4}{2}$, $g\geq 2$ and $\rho(g,2,d)<0$ (see \Cref{irredM2gd}). We also give examples of reducible $\BN{g}{2}{d}$, showing that this bound on $d$ for irreducibility is optimal.

A natural strategy for finding reducible Brill--Noether loci is to look for curves with special geometric conditions that force them to have unexpected linear series.
Recall that an {$r$-dimensional} theta-characteristic on a curve $C$ is a line bundle whose square is the canonical and which has $r+1$ sections.
Theta-characteristics appeared classically associated to theta functions until Mumford provided a purely algebraic framework for their study in \cite{MTheta}.
For a generic curve, a theta-characteristic either has no sections or at most one. 
Therefore, the locus $T^r_g$, of curves of genus $g$ that possess a theta-characteristic of dimension $r$ is a proper subset of ${\mathcal M}_g$ if $r\ge 1$.
Work of Harris \cite{HTheta} led to the expectation that for sufficiently large $g$, $T^r_g$ should have codimension ${r+1 \choose 2}$ in ${\mathcal M}_g$.
Comparing the expected  dimension of the Hilbert scheme with the expected dimension of $T^r_g$, Farkas conjectured in \cite{F} that for $g\geq {r+2 \choose 2}$, $T^r_g$ should have a component of the expected codimension.
This conjecture was proved in \cite{Benzo}. Using limit linear series on chains of elliptic curves, we give here a new proof of the existence of $r$-dimensional theta-characteristics.
\begin{theoremintro}\label[theorem]{thmintro_exthch}
	If $g\ge {r+2\choose 2}$, then $T^r_g$ is non-empty 
	and has one component of the expected codimension $ {r+1\choose 2}$ in ${\mathcal M}_g$.
\end{theoremintro}
As a corollary, we find new reducible Brill--Noether loci.
\begin{theoremintro}\label[theorem]{thmintro_thredMgrd}
	For $r\geq 3$ and $g= {r+2\choose 2}$, the Brill--Noether locus ${\mathcal M}^r_{g,g-1}$ is reducible with one component consisting of curves with a theta-characteristic 
	and at least another component with generic curve not possessing a theta-characteristic of dimension $r$.
\end{theoremintro}
    From the deformation properties of half canonical series, this gives rise to a second component of the Hilbert scheme.
  These components already appear when $r=3$, $g=10$ and they should probably be expected for all values of $r\ge 3$ and relatively small values of $g$.

  Finally, using the more explicit description afforded by using chains of elliptic curves, we obtain a third component of the Brill--Noether locus containing $k$-gonal curves, in addition to the two components found above, for particular values of $g$.
  \begin{theoremintro}\label[theorem]{thmintro_thm:3comp_BNloci}
  	Suppose $r=3,5,7$ or $r\geq 8$, and let $g={r+2 \choose 2}$. Then $\BN{g}{r}{g-1}$ has at least three components.
  \end{theoremintro}

\subsection*{Acknowledgments} We would like to thank Asher Auel, Andrei Bud, Gavril Farkas, Ravi Vakil, and Isabel Vogt for helpful conversations, as well as the anonymous referee for their detailed comments and suggestions.
This paper was started at ``A Panorama of Moduli Spaces" at Goethe University Frankfurt am Main. This research was partially conducted during the time when R.H. was supported by the National Science Foundation under Grant No. DMS-2231565.

\section{Background}\label{sec:Background}

\subsection{Theta-characteristics}

A theta-characteristic on a curve $C$ with dualizing sheaf $\omega_C$ is a line bundle $L$ such that $L^{\otimes 2}\isom \omega_C$.
A theta-characteristic  is called even or odd if $h^0(C,L)$ is even or odd. 
Parity of a theta-characteristic is preserved under deformation, so that every curve of genus $g$ has $2^{g-1}(2^g+1)$ even theta-characteristics and  $2^{g-1}(2^g-1)$ odd ones.
For a theta-characteristic with two or more sections, the kernel of the Petri map is non-empty. 
In particular, as a generic curve of genus $g$ has injective Petri map, a generic curve has theta-characteristics with either one or no sections but no theta-characteristics with two or more sections.
The set of smooth curves that have a theta-characteristic with more than two sections is then a proper subset of ${\mathcal M}_g$.

\begin{defn}\label[definition]{defLocThCh}
	Define $T^r_g\coloneqq\{ [C]\in {\mathcal M}_g \text{ such that } \exists L \text{ on } C, L^2=K_C, h^0(C,L)\ge r+1\}$.
\end{defn}

It is known that $T^1_g$ is an irreducible divisor of ${\mathcal M}_g$, cf.~\cite{thetdiv}.
It was proved by Harris in \cite{HTheta} that the codimension of any component of $T^r_g$ is at most $ \frac {r(r+1)}2={r+1\choose 2}$.
This inequality is an equality if $r=2$ and if $r=3$ for $g$ sufficiently large, as shown in \cite{semican}.
It is not always the case, however, that $\dim T^r_g=3g-3- {r+1\choose 2}$.
 For example, the space of  hyperelliptic curves is $2g-1$-dimensional while every hyperelliptic curve has a theta-characteristic with dimension up to the Clifford index $\lfloor \frac {g-1}2\rfloor $.
 So, in this case, the actual codimension is $g-1$, much smaller than the expected codimension which is of the order of $\frac{g^2-1}4$.
 
In \cite{F}, Farkas proved that for $r\le 11$ and  $g> g(r)$ (with specific ad hoc $g(r)$), there exists one component of $T^r_g$ of the correct dimension.
He then conjectured that a component of the correct dimension should exist for any $r$ and $g\ge {r+2\choose 2}$.
This conjecture was proved in \cite{Benzo} using a smoothing criteria for curves immersed in projective space.
In \Cref{secthch}, we give  a simpler proof of this result.

\subsection{Limit linear series on chains of elliptic curves}
A limit linear series of degree $d$ and dimension $r$ on a curve of compact type consists of a line bundle of degree $d$ and an $r+1$-dimensional space of sections of this line bundle on every component.
The sections need to satisfy that if two components $C, C'$ are glued by identifying the point $P\in C$ to the point $P'\in C'$, 
then the orders of vanishing at $P$ of the space of sections on $C$ and the orders of vanishing at $P'$ of the space of sections on $C'$ can be paired so that corresponding pairs  add to at least $d$,

Choose $g$ elliptic curves $C_1,\dots, C_g$.
On each  curve $C_i$, choose points $P_i, Q_i$.
We glue $C_i$ to $C_{i+1}$ for $i=1,\dots, g-1$, gluing $Q_i$ to $P_{i+1}$.
We will call the resulting curve \[C_1 \cup_{Q_1\sim P_2} C_2 \cup_{Q_2 \sim P_3} C_3\cup\cdots\cup C_{g-1} \cup_{Q_{g-1}\sim P_g} C_g\] a chain of elliptic curves.
For a generic elliptic chain, the components of the sets of linear series correspond to Young tableaux, 
that is, fillings of an $(r+1)\times (g-d+r)$-rectangle with numbers chosen among $1,2,\dots, g$ so that both rows and columns appear in increasing order (see \cite{LT}).
 
We will consider special chains of elliptic curves and limit linear series on them that extend to linear series on nearby non-singular curves.
Given positive integers $g, r, d$, there is a bijection between  fillings of an $(r+1)\times (g-d+r)$ rectangle satisfying the conditions we will list below and limit linear series as follows (see  \cite[Proposition~1.4]{DBN}):
Assume that the chain of elliptic curves is generic except for components   $C_{i_1},\dots,C_{i_e}$  satisfying  $l_1(P_{i_1}-Q_{i_1})=0,\dots, l_e(P_{i_e}-Q_{i_e})=0 $ 
with $l_i$ the smallest positive integer satisfying such a condition.
Then the space of refined limit linear series on the chain are in bijection with the fillings of the rectangle such that:
\begin{enumerate}
	\item[(A1)] both rows and columns are strictly increasing,
	\item[(A2)] only the indices 
	$i_1,\dots, i_e$  may appear more than once  and when they do, the grid distance between the spots corresponding to two such appearances is divisible by $l_i$.
	\item[(A3)] no index could be replaced by one of the $i_k$ while still satisfying the previous two conditions.
\end{enumerate}  
If in addition, the codimension in the space of limit linear series of such chains coincides with $-\rho$,
 then the limit linear series extends to a linear series on nearby curves, as in \cite[Theorem~2.1]{DBN}.

We will operate the other way around to obtain limit linear series on chains of elliptic curves that smooth to nearby curves. Given a filling of the rectangle, the position of the double indices determines which elliptic chains we consider. That is, 
we assume that the elliptic components and nodes are generic except for the torsion required by the repeated indices in the rectangle.
 \begin{figure}[h!]
\begin{tikzpicture}[scale=.45]
\begin{scope}[ ]
				\foreach \x in {0, 1,2,3,4} {\draw[thick] (0,\x) -- (4, \x); }
				\foreach \x in {0, 1,2,3,4} {\draw[thick] (\x,0) -- ( \x,4); }
				
				\node at (.5, 3.5){1}; 
				\node [color=cyan]  at (1.5, 3.5){2}; 
				\node  [color=cyan] at (2.5, 3.5){3}; 
				\node[color=cyan] at (3.5, 3.5){6}; 
				
				\node [color=blue] at (.5, 2.5){2};
				\node at (1.5, 2.5){4}; 
				\node [color=cyan]  at (2.5, 2.5){5}; 
				\node [color=cyan]  at (3.5, 2.5){7}; 
				
				\node [color=blue] at (.5, 1.5){3};
				\node  [color=blue]  at (1.5, 1.5){6}; 
				\node   at (2.5, 1.5){8}; 
				\node [color=cyan]  at (3.5, 1.5){9}; 
				
				\node [color=blue] at (.5, 0.5){5};
				\node  [color=blue]  at (1.5, 0.5){7}; 
				\node   [color=cyan]  at (2.5, 0.5){9}; 
				\node at (3.5, 0.5){10}; 
			\end{scope}
		
\end{tikzpicture}
\caption{  Example of the filling of a rectangle corresponding to a limit linear series (a limit $\g{3}{9}$) on a chain of 10 elliptic curves.}
\label{figex}
\end{figure}
	
For example, for the filling in \Cref{figex}, we have a chain of 10 elliptic curves $C_1\cup\cdots\cup C_{10}$ glued at points $P_i, Q_i\in C_i$ with $Q_i$  glued to $P_{i+1}$ and so  that the points $P_i$ and $Q_i$ are generic for $i=1, 4, 8, 10$. For the repeated indices $2,3,5,6,7,9$, we assume that the points $P_i$ and $Q_i$ satisfy the following torsion conditions: 
on $C_2, C_9$, ${\mathcal O}(2P_i)={\mathcal O}(2Q_i)$, and on $C_3, C_5, C_6, C_7$, ${\mathcal O}(4P_i)={\mathcal O}(4Q_i)$.

The last condition (A3) we impose on the fillings, that no index could be replaced by one of the $i_k$, guarantees that the space of limit linear series on the space of chains of elliptic curves has dimension $-\rho$ and therefore these limit linear series are actually limits of linear series of the same degree and dimension on nearby non-singular curves.

\section{Existence of theta-characteristics with sections}\label{secthch}

We would like to show that there are families of curves with a semicanonical linear series of dimension $r$ near a chain of elliptic curves.
This works for $r=2$ but, unfortunately,  not for $r\ge 3$.

\begin{figure}[h!]
	\begin{tikzpicture}[scale=.45]
		
		\begin{scope}[xshift=0cm]
			\foreach \x in {0, 1,2,3} {\draw[thick] (0,\x) -- (3, \x); }
			\foreach \x in {0, 1,2,3} {\draw[thick] (\x,0) -- ( \x,3); }
			
			\node at (.5, 2.5){1};
			\node  [color=cyan] at (1.5, 2.5){2}; 
			\node [color=cyan]  at (2.5, 2.5){3}; 
			
			\node [color=blue] at (.5, 1.5){2};
			\node  at (1.5, 1.5){4}; 
			\node [color=cyan]  at (2.5, 1.5){5}; 
			
			\node [color=blue] at (.5, 0.5){3};
			\node  [color=blue]  at (1.5, 0.5){5}; 
			\node at (2.5, 0.5){6}; 
		\end{scope}
		
		\begin{scope}[xshift=3.5cm]
			\foreach \x in {0, 1,2,3,4} {\draw[thick] (0,\x) -- (4, \x); }
			\foreach \x in {0, 1,2,3,4} {\draw[thick] (\x,0) -- ( \x,4); }
			
			\node at (.5, 3.5){1}; 
			\node [color=cyan]  at (1.5, 3.5){2}; 
			\node  [color=cyan] at (2.5, 3.5){3}; 
			\node[color=red] at (3.5, 3.5){5}; 
			
			\node [color=blue] at (.5, 2.5){2};
			\node  at (1.5, 2.5){4}; 
			\node [color=red]  at (2.5, 2.5){6}; 
			\node [color=cyan]  at (3.5, 2.5){7}; 
			
			\node [color=blue] at (.5, 1.5){3};
			\node  [color=red]  at (1.5, 1.5){6}; 
			\node at (2.5, 1.5){8}; 
			\node [color=cyan]  at (3.5, 1.5){9}; 
			
			\node [color=red] at (.5, 0.5){5};
			\node  [color=blue]  at (1.5, 0.5){7}; 
			\node [color=blue]  at (2.5, 0.5){9}; 
			\node at (3.5, 0.5){10}; 
		\end{scope}
		
		\begin{scope}[xshift=8.0cm]
			\foreach \x in {0, 1,2,3,4, 5} {\draw[thick] (0,\x) -- (5, \x); }
			\foreach \x in {0, 1,2,3,4,5} {\draw[thick] (\x,0) -- ( \x,5); }
			
			\node at (.5, 4.5){1}; 
			\node [color=cyan]  at (1.5, 4.5){2}; 
			\node  [color=cyan] at (2.5, 4.5){3}; 
			\node[color=cyan] at (3.5, 4.5){4}; 
			\node[color=cyan] at (4.5, 4.5){5}; 
			
			\node [color=blue] at (.5, 3.5){2};
			\node  at (1.5, 3.5){6}; 
			\node  [color=cyan] at (2.5, 3.5){7}; 
			\node [color=cyan]at (3.5, 3.5){8}; 
			\node[color=red] at (4.5, 3.5){10}; 
			
			\node[color=blue] at (.5, 2.5){3}; 
			\node [color=blue] at (1.5, 2.5){7}; 
			\node at (2.5, 2.5){9}; 
			\node [color=red] at (3.5, 2.5){11}; 
			\node[color=cyan] at (4.5, 2.5){12}; 
			
			\node[color=blue] at (.5, 1.5){4}; 
			\node[color=blue] at (1.5, 1.5){8}; 
			\node [color=red] at (2.5, 1.5){11}; 
			\node at (3.5, 1.5){13}; 
			\node [color=cyan]at (4.5, 1.5){14}; 
			
			\node[color=blue] at (.5, .5){5}; 
			\node[color=red] at (1.5, .5){10}; 
			\node[color=blue] at (2.5, .5){12}; 
			\node [color=blue]  at (3.5, .5){14}; 
			\node at (4.5, .5){15}; 
		\end{scope}
		
		\begin{scope}[xshift=13.7cm]
			\foreach \x in {0, 1,2,3,4, 5,6} {\draw[thick] (0,\x) -- (6, \x); }
			\foreach \x in {0, 1,2,3,4,5,6} {\draw[thick] (\x,0) -- ( \x,6); }
			
			\node at (.5, 5.5){1}; 
			\node  [color=cyan]at (1.5, 5.5){2}; 
			\node [color=cyan] at (2.5, 5.5){3}; 
			\node[color=cyan] at (3.5, 5.5){4}; 
			\node[color=cyan] at (4.5, 5.5){5}; 
			\node[color=cyan] at (5.5, 5.5){6}; 
			
			\node [color=blue] at (.5, 4.5){2}; 
			\node at (1.5, 4.5){7}; 
			\node [color=cyan]  at (2.5, 4.5){8}; 
			\node[color=cyan] at (3.5, 4.5){9}; 
			\node[color=cyan] at (4.5, 4.5){10}; 
			\node[color=cyan] at (5.5, 4.5){11}; 
			
			\node [color=blue] at (.5, 3.5){3};
			\node   [color=blue]  at (1.5, 3.5){8}; 
			\node at (2.5, 3.5){12}; 
			\node [color=cyan]at (3.5, 3.5){13}; 
			\node[color=cyan] at (4.5, 3.5){14}; 
			\node[color=red] at (5.5, 3.5){16}; 
			
			\node[color=blue] at (.5, 2.5){4}; 
			\node [color=blue] at (1.5, 2.5){9}; 
			\node [color=blue]  at (2.5, 2.5){13}; 
			\node at (3.5, 2.5){15}; 
			\node[color=red] at (4.5, 2.5){17}; 
			\node[color=cyan] at (5.5, 2.5){18}; 
			
			\node[color=blue] at (.5, 1.5){5}; 
			\node[color=blue] at (1.5, 1.5){10}; 
			\node [color=blue] at (2.5, 1.5){14}; 
			\node [color=red] at (3.5, 1.5){17}; 
			\node at (4.5, 1.5){19}; 
			\node[color=cyan] at (5.5, 1.5){20}; 
			
			\node[color=blue] at (.5, .5){6}; 
			\node[color=blue] at (1.5, .5){11}; 
			\node[color=red] at (2.5, .5){16}; 
			\node  [color=blue] at (3.5, .5){18}; 
			\node [color=blue] at (4.5, .5){20}; 
			\node at (5.5, .5){21}; 
		\end{scope}
		
		\begin{scope}[xshift=20.4cm]
			\foreach \x in {0, 1,2,3,4, 5,6,7} {\draw[thick] (0,\x) -- (7, \x); }
			\foreach \x in {0, 1,2,3,4,5,6,7} {\draw[thick] (\x,0) -- ( \x,7); }
			
			\node at (.5, 6.5){1}; 
			\node  [color=cyan]at (1.5, 6.5){2}; 
			\node [color=cyan] at (2.5, 6.5){3}; 
			\node[color=cyan] at (3.5, 6.5){4}; 
			\node[color=cyan] at (4.5, 6.5){5}; 
			\node[color=cyan] at (5.5, 6.5){6}; 
			\node[color=cyan] at (6.5, 6.5){7}; 
			
			\node  [color=blue] at (.5, 5.5){2}; 
			\node  at (1.5, 5.5){8}; 
			\node [color=cyan] at (2.5, 5.5){9}; 
			\node[color=cyan] at (3.5, 5.5){10}; 
			\node[color=cyan] at (4.5, 5.5){11}; 
			\node[color=cyan] at (5.5, 5.5){12}; 
			\node[color=cyan] at (6.5, 5.5){13}; 
			
			\node [color=blue] at (.5, 4.5){3}; 
			\node [color=blue]  at (1.5, 4.5){9}; 
			\node at (2.5, 4.5){14}; 
			\node[color=cyan] at (3.5, 4.5){15}; 
			\node[color=cyan] at (4.5, 4.5){16}; 
			\node[color=cyan] at (5.5, 4.5){17}; 
			\node[color=cyan] at (6.5, 4.5){18}; 
			
			\node [color=blue] at (.5, 3.5){4};
			\node   [color=blue]  at (1.5, 3.5){10}; 
			\node  [color=blue] at (2.5, 3.5){15}; 
			\node at (3.5, 3.5){19}; 
			\node[color=cyan] at (4.5, 3.5){20}; 
			\node[color=cyan] at (5.5, 3.5){21}; 
			\node[color=red] at (6.5, 3.5){23}; 
			
			\node[color=blue] at (.5, 2.5){5}; 
			\node [color=blue] at (1.5, 2.5){11}; 
			\node [color=blue]  at (2.5, 2.5){16}; 
			\node [color=blue] at (3.5, 2.5){20}; 
			\node at (4.5, 2.5){22}; 
			\node[color=red] at (5.5, 2.5){24}; 
			\node[color=cyan] at (6.5, 2.5){26}; 
			
			\node[color=blue] at (.5, 1.5){6}; 
			\node[color=blue] at (1.5, 1.5){12}; 
			\node [color=blue] at (2.5, 1.5){17}; 
			\node [color=blue] at (3.5, 1.5){21}; 
			\node [color=red] at (4.5, 1.5){24}; 
			\node at (5.5, 1.5){25}; 
			\node[color=cyan] at (6.5, 1.5){27}; 
			
			\node[color=blue] at (.5, .5){7}; 
			\node[color=blue] at (1.5, .5){13}; 
			\node[color=blue] at (2.5, .5){18}; 
			\node  [color=red] at (3.5, .5){23}; 
			\node [color=blue] at (4.5, .5){26}; 
			\node [color=blue] at (5.5, .5){27}; 
			\node at (6.5, .5){28}; 
			
		\end{scope}
		
		\begin{scope}[xshift=28.1cm]
			\foreach \x in {0, 1,2,3,4, 5,6,7,8} {\draw[thick] (0,\x) -- (8, \x); }
			\foreach \x in {0, 1,2,3,4,5,6,7,8} {\draw[thick] (\x,0) -- ( \x,8); }
			
			\node at (.5, 7.5){1}; 
			\node  [color=cyan]at (1.5, 7.5){2}; 
			\node [color=cyan] at (2.5, 7.5){3}; 
			\node[color=cyan] at (3.5, 7.5){4}; 
			\node[color=cyan] at (4.5, 7.5){5}; 
			\node[color=cyan] at (5.5, 7.5){6}; 
			\node[color=cyan] at (6.5, 7.5){7}; 
			\node[color=cyan] at (7.5, 7.5){8}; 
			
			\node  [color=blue] at (.5, 6.5){2}; 
			\node  at (1.5, 6.5){9}; 
			\node [color=cyan] at (2.5, 6.5){10}; 
			\node[color=cyan] at (3.5, 6.5){11}; 
			\node[color=cyan] at (4.5, 6.5){12}; 
			\node[color=cyan] at (5.5, 6.5){13}; 
			\node[color=cyan] at (6.5, 6.5){14}; 
			\node[color=cyan] at (7.5, 6.5){15}; 
			
			\node [color=blue] at (.5, 5.5){3}; 
			\node [color=blue]  at (1.5, 5.5){10}; 
			\node at (2.5, 5.5){16}; 
			\node[color=cyan] at (3.5, 5.5){17}; 
			\node[color=cyan] at (4.5, 5.5){18}; 
			\node[color=cyan] at (5.5, 5.5){19}; 
			\node[color=cyan] at (6.5, 5.5){20}; 
			\node[color=cyan] at (7.5, 5.5){21}; 
			
			\node [color=blue] at (.5, 4.5){4};
			\node   [color=blue]  at (1.5, 4.5){11}; 
			\node  [color=blue] at (2.5, 4.5){17}; 
			\node at (3.5, 4.5){22}; 
			\node[color=cyan] at (4.5, 4.5){23}; 
			\node[color=cyan] at (5.5, 4.5){24}; 
			\node[color=cyan] at (6.5, 4.5){25}; 
			\node[color=cyan] at (7.5, 4.5){26}; 
			
			\node[color=blue] at (.5, 3.5){5}; 
			\node [color=blue] at (1.5, 3.5){12}; 
			\node [color=blue]  at (2.5, 3.5){18}; 
			\node [color=blue] at (3.5, 3.5){23}; 
			\node at (4.5, 3.5){27}; 
			\node[color=cyan] at (5.5, 3.5){28}; 
			\node[color=cyan] at (6.5, 3.5){29}; 
			\node[color=red] at (7.5, 3.5){32}; 
			
			\node[color=blue] at (.5, 2.5){6}; 
			\node[color=blue] at (1.5, 2.5){13}; 
			\node [color=blue] at (2.5, 2.5){19}; 
			\node [color=blue] at (3.5, 2.5){24}; 
			\node [color=blue] at (4.5, 2.5){28}; 
			\node at (5.5, 2.5){30}; 
			\node[color=red] at (6.5, 2.5){31}; 
			\node[color=cyan] at (7.5, 2.5){33}; 
			
			\node[color=blue] at (.5, 1.5){7}; 
			\node[color=blue] at (1.5, 1.5){14}; 
			\node[color=blue] at (2.5, 1.5){20}; 
			\node  [color=blue] at (3.5, 1.5){25}; 
			\node [color=blue] at (4.5, 1.5){29}; 
			\node [color=red] at (5.5, 1.5){32}; 
			\node at (6.5, 1.5){34}; 
			\node[color=cyan] at (7.5, 1.5){35}; 
			
			\node[color=blue] at (.5, .5){8}; 
			\node[color=blue] at (1.5, .5){15}; 
			\node[color=blue] at (2.5, .5){21}; 
			\node  [color=blue] at (3.5, .5){26}; 
			\node [color=red] at (4.5, .5){31}; 
			\node [color=blue] at (5.5, .5){33}; 
			\node [color=blue] at (6.5, .5){35}; 
			\node at (7.5, .5){36}; 
			\end{scope}

\end{tikzpicture}
\caption{Naive fillings of the rectangle when $r=2, 3, 4,5,6, 7$. }
\label{figsqtc}
\end{figure}

As we discussed above, limit linear series on chains of elliptic curves correspond to admissible fillings of an $(r+1)\times (g-d+r)$ rectangle.
For $d=g-1$, the rectangle becomes a square.
The Serre dual linear series corresponds to the transpose  $(g-d+r)\times (r+1)$ rectangle, 
in particular, the linear series is half canonical if the filling of the square is symmetric with respect to the diagonal.

One way to naively construct such a half-canonical filling is to fill each row  and column consecutively so that the filling is symmetric. 
These naive fillings will never amount to an admissible filling: when there are two squares filled with the same number $x$ on either side of the main diagonal that touch at one corner (thus specifying $2$-torsion),
the squares of that antidiagonal must also be filled with $x$ (as they specify even torsion) unless there is a square below or to the right that has already been filled with a number $<x$. 
One can focus this problem to only occur for two consecutive entries, as shown in \Cref{figsqtc}, 
where the entries that make the filling not admissible are highlighted in red and should be replaced with the entry specifying the $2$-torsion. 
These fillings match our construction in \Cref{exthch}.

For example, when $r=4$, in \Cref{figsqtc}, the $4$'s cannot be replaced with $7$'s because there is a $5$ under and below the two $4$'s and $5<7$.
On the other hand, the 10 could be replaced with an $11$ as there are no entries smaller than $11$ to the right or below the $11$.
However, while replacing the $10$'s by $11$'s gives an admissible filling, it also removes any torsion conditions on the tenth elliptic curve, 
allowing us to take arbitrary points as nodes on the tenth elliptic curve and still obtain a half-canonical linear series on a chain of 15 elliptic curves with 5 sections. 
Therefore, the codimension of the set of curves of genus 15 with a half-canonical limit linear series of dimension 4 inside the elliptic chains is 9 instead of the ${4+1\choose 2}=10$ expected.
Thus we have no guarantee that this half-canonical series will smooth to a non-singular curve.
 
Our approach is to replace the two (consecutive) elliptic curves corresponding to the entries of this antidiagonal by a curve of genus $2$, 
with Weierstrass points as nodes, thus eliminating this larger family of limit linear series on the chain of curves associated to these naive fillings.

\begin{theorem}\label[theorem]{exthch} If $g\ge {r+2\choose 2}$, then $T^r_g$ is non-empty 
and has one component of the expected codimension $ {r+1\choose 2}$ in ${\mathcal M}_g$.
The kernel of the Petri map for a generic point on this component has dimension  $ {r+1\choose 2}$.
\end{theorem}
\begin{proof}  We first consider the case  $g= {r+2\choose 2}$.
We construct a filling of the $(r+1)\times(g-d+r)$ rectangle for a semicanonical linear series of dimension $r$.
	
We fill the rectangle matching the elements on the first row with those on the first column, then the remaining of the second row matching with the remaining of the second column 
and so on  with one exception:
Instead of completing row and column $r-2$, we leave the last spot of each of them empty and fill first the spot $(r-1,r-1)$.
 Then we  resume the filling of the last triangle on the right bottom corner (see \Cref{figsqtc}).
The only problematic components in such a filling are those occupying the spots $(r-2,r+1), (r-1, r), (r,r-1), (r+1,r-2)$ which correspond to two consecutive elliptic curves (see \Cref{figsqthc2}).
We will replace this pair of elliptic curves with a single curve $D$ of genus two with a point  $P\in D$ glued to $Q_{g-6}\in C_{g-6}$ in the chain 
and  a point $Q\in D$ glued to  $P_{g-3}\in C_{g-3}$.

\begin{figure}[h!]
\begin{tikzpicture}[scale=.45]

\begin{scope}[xshift=0cm]
\foreach \x in {0, 1,2,3} {\draw[thick] (0,\x) -- (3, \x); }
\foreach \x in {0, 1,2,3} {\draw[thick] (\x,0) -- ( \x,3); }

\node at (.5, 2.5){1};
\node  [color=cyan] at (1.5, 2.5){2}; 
\node [color=cyan]  at (2.5, 2.5){3}; 

\node [color=blue] at (.5, 1.5){2};
\node  at (1.5, 1.5){4}; 
\node [color=cyan]  at (2.5, 1.5){5}; 

\node [color=blue] at (.5, 0.5){3};
\node  [color=blue]  at (1.5, 0.5){5}; 
\node at (2.5, 0.5){6}; 
\end{scope}

\begin{scope}[xshift=3.6cm]
\foreach \x in {0, 1,2,3,4} {\draw[thick] (0,\x) -- (4, \x); }
\foreach \x in {0, 1,2,3,4} {\draw[thick] (\x,0) -- ( \x,4); }

\node at (.5, 3.5){1}; 
\node [color=cyan]  at (1.5, 3.5){2}; 
\node  [color=cyan] at (2.5, 3.5){3}; 
\node[color=cyan] at (3.5, 3.5){5}; 

\node [color=blue] at (.5, 2.5){2};
\node  at (1.5, 2.5){4}; 
\node [color=cyan]  at (2.5, 2.5){6}; 
\node [color=cyan]  at (3.5, 2.5){7}; 

\node [color=blue] at (.5, 1.5){3};
\node  [color=blue]  at (1.5, 1.5){6}; 
\node at (2.5, 1.5){8}; 
\node [color=cyan]  at (3.5, 1.5){9};

\node [color=blue] at (.5, 0.5){5};
\node  [color=blue]  at (1.5, 0.5){7}; 
\node [color=blue]  at (2.5, 0.5){9}; 
\node at (3.5, 0.5){10}; 
\fill [color=gray](0,0)--(1,0)--(1,1)--(2,1)--(2,2)--(3,2)--(3,3)--(4,3)--(4,4)--(3,4)--(3,3)--(2,3)--(2,2)--(1,2)--(1,1)--(0,1)--(0,0);
\end{scope}

\begin{scope}[xshift=8.2cm]
\foreach \x in {0, 1,2,3,4, 5} {\draw[thick] (0,\x) -- (5, \x); }
\foreach \x in {0, 1,2,3,4,5} {\draw[thick] (\x,0) -- ( \x,5); }

\node at (.5, 4.5){1}; 
\node [color=cyan]  at (1.5, 4.5){2}; 
\node  [color=cyan] at (2.5, 4.5){3}; 
\node[color=cyan] at (3.5, 4.5){4}; 
\node[color=cyan] at (4.5, 4.5){5}; 

\node [color=blue] at (.5, 3.5){2};
\node  at (1.5, 3.5){6}; 
\node  [color=cyan] at (2.5, 3.5){7}; 
\node [color=cyan]at (3.5, 3.5){8}; 
\node[color=cyan] at (4.5, 3.5){10}; 

\node[color=blue] at (.5, 2.5){3}; 
\node [color=blue] at (1.5, 2.5){7}; 
\node at (2.5, 2.5){9}; 
\node [color=cyan] at (3.5, 2.5){11}; 
\node[color=cyan] at (4.5, 2.5){12}; 

\node[color=blue] at (.5, 1.5){4}; 
\node[color=blue] at (1.5, 1.5){8}; 
\node [color=blue] at (2.5, 1.5){11}; 
\node at (3.5, 1.5){13}; 
\node [color=cyan]at (4.5, 1.5){14}; 

\node[color=blue] at (.5, .5){5}; 
\node[color=blue] at (1.5, .5){10}; 
\node[color=blue] at (2.5, .5){12}; 
\node [color=blue]  at (3.5, .5){14}; 
\node at (4.5, .5){15}; 
\fill [color=gray](1,0)--(2,0)--(2,1)--(3,1)--(3,2)--(4,2)--(4,3)--(5,3)--(5,4)--(4,4)--(4,3)--(3,3)--(3,2)--(2,2)--(2,1)--(1,1)--(1,0);
\end{scope}

\begin{scope}[xshift=13.8cm]
\foreach \x in {0, 1,2,3,4, 5,6} {\draw[thick] (0,\x) -- (6, \x); }
\foreach \x in {0, 1,2,3,4,5,6} {\draw[thick] (\x,0) -- ( \x,6); }

\node at (.5, 5.5){1}; 
\node  [color=cyan]at (1.5, 5.5){2}; 
\node [color=cyan] at (2.5, 5.5){3}; 
\node[color=cyan] at (3.5, 5.5){4}; 
\node[color=cyan] at (4.5, 5.5){5}; 
\node[color=cyan] at (5.5, 5.5){6}; 

\node [color=blue] at (.5, 4.5){2}; 
\node at (1.5, 4.5){7}; 
\node [color=cyan]  at (2.5, 4.5){8}; 
\node[color=cyan] at (3.5, 4.5){9}; 
\node[color=cyan] at (4.5, 4.5){10}; 
\node[color=cyan] at (5.5, 4.5){11}; 

\node [color=blue] at (.5, 3.5){3};
\node   [color=blue]  at (1.5, 3.5){8}; 
\node at (2.5, 3.5){12}; 
\node [color=cyan]at (3.5, 3.5){13}; 
\node[color=cyan] at (4.5, 3.5){14}; 
\node[color=cyan] at (5.5, 3.5){16}; 

\node[color=blue] at (.5, 2.5){4}; 
\node [color=blue] at (1.5, 2.5){9}; 
\node [color=blue]  at (2.5, 2.5){13}; 
\node at (3.5, 2.5){15}; 
\node[color=cyan] at (4.5, 2.5){17}; 
\node[color=cyan] at (5.5, 2.5){18}; 

\node[color=blue] at (.5, 1.5){5}; 
\node[color=blue] at (1.5, 1.5){10}; 
\node [color=blue] at (2.5, 1.5){14}; 
\node [color=blue] at (3.5, 1.5){17}; 
\node at (4.5, 1.5){19}; 
\node[color=cyan] at (5.5, 1.5){20}; 

\node[color=blue] at (.5, .5){6}; 
\node[color=blue] at (1.5, .5){11}; 
\node[color=blue] at (2.5, .5){16}; 
\node  [color=blue] at (3.5, .5){18}; 
\node [color=blue] at (4.5, .5){20}; 
\node at (5.5, .5){21}; 
\fill [color=gray](2,0)--(3,0)--(3,1)--(4,1)--(4,2)--(5,2)--(5,3)--(6,3)--(6,4)--(5,4)--(5,3)--(4,3)--(4,2)--(3,2)--(3,1)--(2,1)--(2,0);

\end{scope}

\begin{scope}[xshift=20.5cm]
\foreach \x in {0, 1,2,3,4, 5,6,7} {\draw[thick] (0,\x) -- (7, \x); }
\foreach \x in {0, 1,2,3,4,5,6,7} {\draw[thick] (\x,0) -- ( \x,7); }

\node at (.5, 6.5){1}; 
\node  [color=cyan]at (1.5, 6.5){2}; 
\node [color=cyan] at (2.5, 6.5){3}; 
\node[color=cyan] at (3.5, 6.5){4}; 
\node[color=cyan] at (4.5, 6.5){5}; 
\node[color=cyan] at (5.5, 6.5){6}; 
\node[color=cyan] at (6.5, 6.5){7}; 

\node  [color=blue] at (.5, 5.5){2}; 
\node  at (1.5, 5.5){8}; 
\node [color=cyan] at (2.5, 5.5){9}; 
\node[color=cyan] at (3.5, 5.5){10}; 
\node[color=cyan] at (4.5, 5.5){11}; 
\node[color=cyan] at (5.5, 5.5){12}; 
\node[color=cyan] at (6.5, 5.5){13}; 

\node [color=blue] at (.5, 4.5){3}; 
\node [color=blue]  at (1.5, 4.5){9}; 
\node at (2.5, 4.5){14}; 
\node[color=cyan] at (3.5, 4.5){15}; 
\node[color=cyan] at (4.5, 4.5){16}; 
\node[color=cyan] at (5.5, 4.5){17}; 
\node[color=cyan] at (6.5, 4.5){18}; 

\node [color=blue] at (.5, 3.5){4};
\node   [color=blue]  at (1.5, 3.5){10}; 
\node  [color=blue] at (2.5, 3.5){15}; 
\node at (3.5, 3.5){19}; 
\node[color=cyan] at (4.5, 3.5){20}; 
\node[color=cyan] at (5.5, 3.5){21}; 
\node[color=cyan] at (6.5, 3.5){23}; 

\node[color=blue] at (.5, 2.5){5}; 
\node [color=blue] at (1.5, 2.5){11}; 
\node [color=blue]  at (2.5, 2.5){16}; 
\node [color=blue] at (3.5, 2.5){20}; 
\node at (4.5, 2.5){22}; 
\node[color=cyan] at (5.5, 2.5){24}; 
\node[color=cyan] at (6.5, 2.5){26}; 

\node[color=blue] at (.5, 1.5){6}; 
\node[color=blue] at (1.5, 1.5){12}; 
\node [color=blue] at (2.5, 1.5){17}; 
\node [color=blue] at (3.5, 1.5){21}; 
\node [color=blue] at (4.5, 1.5){24}; 
\node at (5.5, 1.5){25}; 
\node[color=cyan] at (6.5, 1.5){27}; 

\node[color=blue] at (.5, .5){7}; 
\node[color=blue] at (1.5, .5){13}; 
\node[color=blue] at (2.5, .5){18}; 
\node  [color=blue] at (3.5, .5){23}; 
\node [color=blue] at (4.5, .5){26}; 
\node [color=blue] at (5.5, .5){27}; 
\node at (6.5, .5){28}; 
\fill [color=gray](3,0)--(4,0)--(4,1)--(5,1)--(5,2)--(6,2)--(6,3)--(7,3)--(7,4)--(6,4)--(6,3)--(5,3)--(5,2)--(4,2)--(4,1)--(3,1)--(3,0);

\end{scope}

\begin{scope}[xshift=28.4cm]
\foreach \x in {0, 1,2,3,4, 5,6,7,8} {\draw[thick] (0,\x) -- (8, \x); }
\foreach \x in {0, 1,2,3,4,5,6,7,8} {\draw[thick] (\x,0) -- ( \x,8); }

\node at (.5, 7.5){1}; 
\node  [color=cyan]at (1.5, 7.5){2}; 
\node [color=cyan] at (2.5, 7.5){3}; 
\node[color=cyan] at (3.5, 7.5){4}; 
\node[color=cyan] at (4.5, 7.5){5}; 
\node[color=cyan] at (5.5, 7.5){6}; 
\node[color=cyan] at (6.5, 7.5){7}; 
\node[color=cyan] at (7.5, 7.5){8}; 

\node  [color=blue] at (.5, 6.5){2}; 
\node  at (1.5, 6.5){9}; 
\node [color=cyan] at (2.5, 6.5){10}; 
\node[color=cyan] at (3.5, 6.5){11}; 
\node[color=cyan] at (4.5, 6.5){12}; 
\node[color=cyan] at (5.5, 6.5){13}; 
\node[color=cyan] at (6.5, 6.5){14}; 
\node[color=cyan] at (7.5, 6.5){15}; 

\node [color=blue] at (.5, 5.5){3}; 
\node [color=blue]  at (1.5, 5.5){10}; 
\node at (2.5, 5.5){16}; 
\node[color=cyan] at (3.5, 5.5){17}; 
\node[color=cyan] at (4.5, 5.5){18}; 
\node[color=cyan] at (5.5, 5.5){19}; 
\node[color=cyan] at (6.5, 5.5){20}; 
\node[color=cyan] at (7.5, 5.5){21}; 

\node [color=blue] at (.5, 4.5){4};
\node   [color=blue]  at (1.5, 4.5){11}; 
\node  [color=blue] at (2.5, 4.5){17}; 
\node at (3.5, 4.5){22}; 
\node[color=cyan] at (4.5, 4.5){23}; 
\node[color=cyan] at (5.5, 4.5){24}; 
\node[color=cyan] at (6.5, 4.5){25}; 
\node[color=cyan] at (7.5, 4.5){26}; 

\node[color=blue] at (.5, 3.5){5}; 
\node [color=blue] at (1.5, 3.5){12}; 
\node [color=blue]  at (2.5, 3.5){18}; 
\node [color=blue] at (3.5, 3.5){23}; 
\node at (4.5, 3.5){27}; 
\node[color=cyan] at (5.5, 3.5){28}; 
\node[color=cyan] at (6.5, 3.5){29}; 
\node[color=cyan] at (7.5, 3.5){32}; 

\node[color=blue] at (.5, 2.5){6}; 
\node[color=blue] at (1.5, 2.5){13}; 
\node [color=blue] at (2.5, 2.5){19}; 
\node [color=blue] at (3.5, 2.5){24}; 
\node [color=blue] at (4.5, 2.5){28}; 
\node at (5.5, 2.5){30}; 
\node[color=cyan] at (6.5, 2.5){31}; 
\node[color=cyan] at (7.5, 2.5){33}; 

\node[color=blue] at (.5, 1.5){7}; 
\node[color=blue] at (1.5, 1.5){14}; 
\node[color=blue] at (2.5, 1.5){20}; 
\node  [color=blue] at (3.5, 1.5){25}; 
\node [color=blue] at (4.5, 1.5){29}; 
\node [color=blue] at (5.5, 1.5){32}; 
\node at (6.5, 1.5){34}; 
\node[color=cyan] at (7.5, 1.5){35}; 

\node[color=blue] at (.5, .5){8}; 
\node[color=blue] at (1.5, .5){15}; 
\node[color=blue] at (2.5, .5){21}; 
\node  [color=blue] at (3.5, .5){26}; 
\node [color=blue] at (4.5, .5){31}; 
\node [color=blue] at (5.5, .5){33}; 
\node [color=blue] at (6.5, .5){35}; 
\node at (7.5, .5){36}; 
\fill [color=gray](4,0)--(5,0)--(5,1)--(6,1)--(6,2)--(7,2)--(7,3)--(8,3)--(8,4)--(7,4)--(7,3)--(6,3)--(6,2)--(5,2)--(5,1)--(4,1)--(4,0);
\end{scope}

\end{tikzpicture}
\caption{  Replacing a pair of elliptic curves with a curve of genus two. }
\label{figsqthc2}
\end{figure}

To obtain a half-canonical  limit linear series on the new chain, we need a line bundle $L$ on $D$ of degree $g-1$ such that $L^2=K_D(2(g-6)P+8Q)$ and an 
$(r+1) $-dimensional space of sections  with vanishing at $P, Q$ respectively
\[ \begin{pmatrix} g-r-7&g-r-6&\dots&g-12&g-11&g-9&g-7&g-5&g-3\\r+4&r+3&\dots&9&8&7&5&3&1\end{pmatrix}\]
In particular, we need sections $s_1, s_3, s_5, s_7$ such that $s_i$ vanishes at $Q$ with order at least $i$ and at $P$ with order at least $g-i-2$.
There must exist points $R_1, R_3, R_5, R_7$ such that 
\[ L={\mathcal O}(R_1+Q+(g-3)P)={\mathcal O}(R_3+3Q+(g-5)P)={\mathcal O}(R_5+5Q+(g-7)P)={\mathcal O}(R_7+7Q+(g-9)P)\]
Then,
\[ 2P-2Q\equiv R_3-R_1\equiv R_5-R_3\equiv R_7-R_5\]
so that 
\[  2R_3\equiv R_5+R_1, 2R_5\equiv R_3+R_7\]
Then, either $R_3=R_1=R_5$ or  $R_3$ is a Weierstrass point and  $R_1,R_5$ are conjugate under the hyperelliptic involution. 
 Similarly, either  $R_5=R_3=R_7$ or  $R_5$ is a Weierstrass and $R_3,R_7$ are conjugate under the hyperelliptic involution. 
On the other hand, from $L^2=K_D(2(g-6)P+8Q)$, we obtain 
\[ K_D={\mathcal O}(R_1+R_7)={\mathcal O}(R_3+R_5)\]
which implies that $R_1, R_7$ are conjugate in the hyperelliptic involution and so are $R_3, R_5$.
Therefore, all the $R_i$ are identical and are  a Weierstrass point.
Then, $ 2P\equiv 2Q$, so that $P, Q$ are also Weierstrass points.
This imposes two conditions on the pointed curve $(D, P, Q)$, and gives the expected codimension on the chain of $g-2$ elliptic curves and $D$.

Conversely, given  chains of $g-6$  and 4 elliptic curves with torsion conditions between the nodes as determined by the filling and a curve $D$ of genus 2 with 2 Weierstrass points $P, Q$,
we can form a chain of genus $g$ attaching $D$ by $P$ to the $(g-6)$ chain and by $Q$ to the 4-chain.
Take on the elliptic curves  the line bundles and spaces of sections determined by the fillings. 
On $D$, choose a Weierstrass point $R$ different from $P, Q$, and take the line bundle 
\[ L={\mathcal O}(R+Q+(g-3)P)={\mathcal O}(R+3Q+(g-5)P)={\mathcal O}(R+5Q+(g-7)P)={\mathcal O}(R+7Q+(g-9)P)\]
Then, this line bundle has an $(r+1)$-dimensional space of sections with the vanishing sequence listed above namely the subspace of sections vanishing with order $g-r-7$ at $P$ and 8 at $Q$
together with 4 copies of the $g^1_2$:
\[ [H^0(D, L(-(g-r-7)P-8Q))+((g-r-7)P+8Q)]\oplus [4g^1_2+((g-9)P+Q)].\] 

The surjectivity of the Petri map is proved using the methods in \cite{LOTZ1}, \cite{LOTZ2} and is similar to the proof of \cite[Proposition~1.8]{Ramif} and in \cite[Theorem~2.1]{DBN}. 
One constructs limit sections of the given linear series $s_0,\dots, s_r$ and of the Serre dual $t_0,\dots t_r$. 
One then shows that the limit sections $s_it_j, i\le j$ are independent limit sections of the canonical limit linear series.
In order to prove the independence of these sections, one considers a distribution of the total degree $2g-2$ among the $g$ components. 
In  \cite{Ramif} and \cite{DBN}, the author chose degree 1 on the first and last components and degree 2 on the remaining ones.
Here, due to the presence of the genus 2 curve replacing the elliptic curves $C_{g-5},C_{g-4}$, we choose degree 1 on the curves $C_1, C_g$, 
degree 4 on $C_{g-3}$  and degree 2 on the remaining components. 

\medskip

If  $g> {r+2\choose 2}$, define $g_0=  {r+2\choose 2}$. Consider a chain of $g_0-2$ elliptic curves and one curve of genus 2 as in the above construction.
Denote by $(L_i, V_i)$ the corresponding line bundles and spaces of sections.
Glue $Q_{g_0}$ to a generic chain of $g-g_0$ elliptic curves. 
We now describe a semicanonical limit linear series with $r+1$ sections on the whole chain of $g$ elliptic curves. 
On the curves $C_i,\ i=1\dots, g_0$, take the line bundles and spaces of sections $(L_i((g-g_0)Q_i), V_i +((g-g_0)Q_i))$.
On the curves $C_i,\ i= g_0, \dots,g$, choose a non-trivial line bundle  $L_i$  of order 2 and take the line bundle $L_i(( i-1)P_i+(g-i)Q_i)$.
Take the space of sections of $L_i(( i-1)P_i+(g-i)Q_i)$ with vanishing at $P_i$ being $(i-r-2,  i-r-1,\dots, i-3, i-2)$ and at $Q_i$ being $(g-i+r,  g-i+r-1,\dots, g-i+1, g-i)$.
\end{proof}

\begin{defn}
	We will write $ET^r_g$ to denote the component of $T^r_g$ constructed in \Cref{exthch}.
\end{defn}

\section{Reducible Brill--Noether Loci}

We construct two different families of curves of genus $g$ possessing a $g^r_{g-1}$ so that for one family the $g^r_{g-1}$ is half canonical, and for the other it is not. We first set some notation.

For $r\geq 3$ and $g={r+2 \choose 2}$, we give asymmetric admissible fillings of a $(r+1)\times (r+1)$ square. 
These fillings are obtained by filling the square by South-West to North-East diagonals so that the second half of each diagonal repeats the entries of the first half.
 The first few fillings are given in \Cref{fig_Nfilling}. 
 These fillings are admissible.
 Both rows and columns are in increasing order. 
 Repeated indices are at even distance and there is no third empty spot on that diagonal at the same distance as the pair that have the same index. 
 The only other empty spots when placing a new index are on the next diagonal that are at odd distance from the spots in the previous one.
 Moreover the construction can be made symmetric if we were to start with the index $g$ and go down to $1$, filling from the bottom right corner instead. 
 Therefore, it is also not possible to replace one index already in the table with a larger one to still obtain an admissible filling.

\begin{figure}[h!]
	\begin{tikzpicture}[scale=.45]
		
		\begin{scope}[ ]
			\foreach \x in {0, 1,2,3,4} {\draw[thick] (0,\x) -- (4, \x); }
			\foreach \x in {0, 1,2,3,4} {\draw[thick] (\x,0) -- ( \x,4); }
			
			\node at (.5, 3.5){1}; 
			\node [color=cyan]  at (1.5, 3.5){2}; 
			\node  [color=cyan] at (2.5, 3.5){3}; 
			\node[color=cyan] at (3.5, 3.5){6}; 
			
			\node [color=blue] at (.5, 2.5){2};
			\node at (1.5, 2.5){4}; 
			\node [color=cyan]  at (2.5, 2.5){5}; 
			\node [color=cyan]  at (3.5, 2.5){7}; 
			
			\node [color=blue] at (.5, 1.5){3};
			\node  [color=blue]  at (1.5, 1.5){6}; 
			\node   at (2.5, 1.5){8}; 
			\node [color=cyan]  at (3.5, 1.5){9}; 
			
			\node [color=blue] at (.5, 0.5){5};
			\node  [color=blue]  at (1.5, 0.5){7}; 
			\node   [color=cyan]  at (2.5, 0.5){9}; 
			\node at (3.5, 0.5){10}; 
		\end{scope}
		
		\begin{scope}[xshift=5cm]
			\foreach \x in {0, 1,2,3,4, 5} {\draw[thick] (0,\x) -- (5, \x); }
			\foreach \x in {0, 1,2,3,4,5} {\draw[thick] (\x,0) -- ( \x,5); }
			
			\node at (.5, 4.5){1}; 
			\node  [color=cyan]  at (1.5, 4.5){2}; 
			\node  [color=cyan] at (2.5, 4.5){3}; 
			\node[color=cyan] at (3.5, 4.5){6}; 
			\node[color=cyan] at (4.5, 4.5){8}; 
			
			\node [color=blue] at (.5, 3.5){2};
			\node    at (1.5, 3.5){4}; 
			\node [color=cyan] at (2.5, 3.5){5}; 
			\node [color=cyan] at (3.5, 3.5){7}; 
			\node[color=cyan] at (4.5, 3.5){11}; 
			
			\node[color=blue] at (.5, 2.5){3}; 
			\node [color=blue] at (1.5, 2.5){6}; 
			\node      at (2.5, 2.5){9}; 
			\node [color=cyan] at (3.5, 2.5){10}; 
			\node[color=cyan] at (4.5, 2.5){12}; 
			
			\node[color=blue] at (.5, 1.5){5}; 
			\node[color=blue] at (1.5, 1.5){8}; 
			\node   [color=blue] at (2.5, 1.5){11}; 
			\node   at (3.5, 1.5){13}; 
			\node [color=cyan]at (4.5, 1.5){14}; 
			
			\node[color=blue] at (.5, .5){7}; 
			\node[color=blue] at (1.5, .5){10}; 
			\node [color=blue] at (2.5, .5){12}; 
			\node   [color=blue]  at (3.5, .5){14}; 
			\node at (4.5, .5){15}; 
		\end{scope}

		\begin{scope}[xshift=11cm]
			\foreach \x in {0, 1,2,3,4, 5,6} {\draw[thick] (0,\x) -- (6, \x); }
			\foreach \x in {0, 1,2,3,4,5,6} {\draw[thick] (\x,0) -- ( \x,6); }
			
			\node at (.5, 5.5){1}; 
			\node  [color=cyan]at (1.5, 5.5){2}; 
			\node [color=cyan] at (2.5, 5.5){3}; 
			\node[color=cyan] at (3.5, 5.5){6}; 
			\node[color=cyan] at (4.5, 5.5){8}; 
			\node[color=cyan] at (5.5, 5.5){12}; 
			
			\node [color=blue] at (.5, 4.5){2}; 
			\node at (1.5, 4.5){4}; 
			\node [color=cyan]  at (2.5, 4.5){5}; 
			\node[color=cyan] at (3.5, 4.5){7}; 
			\node[color=cyan] at (4.5, 4.5){11}; 
			\node[color=cyan] at (5.5, 4.5){14}; 
			
			\node [color=blue] at (.5, 3.5){3};
			\node   [color=blue]  at (1.5, 3.5){6}; 
			\node at (2.5, 3.5){9}; 
			\node [color=cyan]at (3.5, 3.5){10}; 
			\node[color=cyan] at (4.5, 3.5){13}; 
			\node[color=cyan] at (5.5, 3.5){17}; 
			
			\node[color=blue] at (.5, 2.5){5}; 
			\node [color=blue] at (1.5, 2.5){8}; 
			\node [color=blue]  at (2.5, 2.5){12}; 
			\node at (3.5, 2.5){15}; 
			\node[color=cyan] at (4.5, 2.5){16}; 
			\node[color=cyan] at (5.5, 2.5){18}; 
			
			\node[color=blue] at (.5, 1.5){7}; 
			\node[color=blue] at (1.5, 1.5){11}; 
			\node [color=blue] at (2.5, 1.5){14}; 
			\node [color=blue] at (3.5, 1.5){17}; 
			\node at (4.5, 1.5){19}; 
			\node[color=cyan] at (5.5, 1.5){20}; 
			
			\node[color=blue] at (.5, .5){10}; 
			\node[color=blue] at (1.5, .5){13}; 
			\node[color=blue] at (2.5, .5){16}; 
			\node  [color=blue] at (3.5, .5){18}; 
			\node [color=blue] at (4.5, .5){20}; 
			\node at (5.5, .5){21}; 
		\end{scope}
		
		\begin{scope}[xshift=18cm]
			\foreach \x in {0, 1,2,3,4, 5,6,7} {\draw[thick] (0,\x) -- (7, \x); }
			\foreach \x in {0, 1,2,3,4,5,6,7} {\draw[thick] (\x,0) -- ( \x,7); }
			
			\node at (.5, 6.5){1}; 
			\node  [color=cyan]at (1.5, 6.5){2}; 
			\node [color=cyan] at (2.5, 6.5){3}; 
			\node[color=cyan] at (3.5, 6.5){6}; 
			\node[color=cyan] at (4.5, 6.5){8}; 
			\node[color=cyan] at (5.5, 6.5){12}; 
			\node[color=cyan] at (6.5, 6.5){15}; 
			
			\node  [color=blue] at (.5, 5.5){2}; 
			\node  at (1.5, 5.5){4}; 
			\node [color=cyan] at (2.5, 5.5){5}; 
			\node[color=cyan] at (3.5, 5.5){7}; 
			\node[color=cyan] at (4.5, 5.5){11}; 
			\node[color=cyan] at (5.5, 5.5){14}; 
			\node[color=cyan] at (6.5, 5.5){19};

			\node [color=blue] at (.5, 4.5){3}; 
			\node [color=blue]  at (1.5, 4.5){6}; 
			\node at (2.5, 4.5){9}; 
			\node[color=cyan] at (3.5, 4.5){10}; 
			\node[color=cyan] at (4.5, 4.5){13}; 
			\node[color=cyan] at (5.5, 4.5){18}; 
			\node[color=cyan] at (6.5, 4.5){21}; 
			
			\node [color=blue] at (.5, 3.5){5};
			\node   [color=blue]  at (1.5, 3.5){8}; 
			\node  [color=blue] at (2.5, 3.5){12}; 
			\node at (3.5, 3.5){16}; 
			\node[color=cyan] at (4.5, 3.5){17}; 
			\node[color=cyan] at (5.5, 3.5){20}; 
			\node[color=cyan] at (6.5, 3.5){24}; 
			
			\node[color=blue] at (.5, 2.5){7}; 
			\node [color=blue] at (1.5, 2.5){11}; 
			\node [color=blue]  at (2.5, 2.5){15}; 
			\node [color=blue] at (3.5, 2.5){19}; 
			\node at (4.5, 2.5){22}; 
			\node[color=cyan] at (5.5, 2.5){23}; 
			\node[color=cyan] at (6.5, 2.5){25}; 
			
			\node[color=blue] at (.5, 1.5){10}; 
			\node[color=blue] at (1.5, 1.5){14}; 
			\node [color=blue] at (2.5, 1.5){18}; 
			\node [color=blue] at (3.5, 1.5){21}; 
			\node [color=blue] at (4.5, 1.5){24}; 
			\node at (5.5, 1.5){26}; 
			\node[color=cyan] at (6.5, 1.5){27}; 
			
			\node[color=blue] at (.5, .5){13}; 
			\node[color=blue] at (1.5, .5){17}; 
			\node[color=blue] at (2.5, .5){20}; 
			\node  [color=blue] at (3.5, .5){23}; 
			\node [color=blue] at (4.5, .5){25}; 
			\node [color=blue] at (5.5, .5){27}; 
			\node at (6.5, .5){28}; 
			
		\end{scope}
		
		\begin{scope}[xshift=26cm]
			\foreach \x in {0, 1,2,3,4, 5,6,7,8} {\draw[thick] (0,\x) -- (8, \x); }
			\foreach \x in {0, 1,2,3,4,5,6,7,8} {\draw[thick] (\x,0) -- ( \x,8); }
			
			\node at (.5, 7.5){1}; 
			\node  [color=cyan]at (1.5, 7.5){2}; 
			\node [color=cyan] at (2.5, 7.5){3}; 
			\node[color=cyan] at (3.5, 7.5){6}; 
			\node[color=cyan] at (4.5, 7.5){8}; 
			\node[color=cyan] at (5.5, 7.5){12}; 
			\node[color=cyan] at (6.5, 7.5){15}; 
			\node[color=cyan] at (7.5, 7.5){20}; 
			
			\node  [color=blue] at (.5, 6.5){2}; 
			\node  at (1.5, 6.5){4}; 
			\node [color=cyan] at (2.5, 6.5){5}; 
			\node[color=cyan] at (3.5, 6.5){7}; 
			\node[color=cyan] at (4.5, 6.5){11}; 
			\node[color=cyan] at (5.5, 6.5){14}; 
			\node[color=cyan] at (6.5, 6.5){19}; 
			\node[color=cyan] at (7.5, 6.5){23}; 
			
			\node [color=blue] at (.5, 5.5){3}; 
			\node [color=blue]  at (1.5, 5.5){6}; 
			\node at (2.5, 5.5){9}; 
			\node[color=cyan] at (3.5, 5.5){10}; 
			\node[color=cyan] at (4.5, 5.5){13}; 
			\node[color=cyan] at (5.5, 5.5){18}; 
			\node[color=cyan] at (6.5, 5.5){22}; 
			\node[color=cyan] at (7.5, 5.5){27}; 
			
			\node [color=blue] at (.5, 4.5){5};
			\node   [color=blue]  at (1.5, 4.5){8}; 
			\node  [color=blue] at (2.5, 4.5){12}; 
			\node at (3.5, 4.5){16}; 
			\node[color=cyan] at (4.5, 4.5){17}; 
			\node[color=cyan] at (5.5, 4.5){21}; 
			\node[color=cyan] at (6.5, 4.5){26}; 
			\node[color=cyan] at (7.5, 4.5){29}; 
			
			\node[color=blue] at (.5, 3.5){7}; 
			\node [color=blue] at (1.5, 3.5){11}; 
			\node [color=blue]  at (2.5, 3.5){15}; 
			\node [color=blue] at (3.5, 3.5){20}; 
			\node at (4.5, 3.5){24}; 
			\node[color=cyan] at (5.5, 3.5){25}; 
			\node[color=cyan] at (6.5, 3.5){28}; 
			\node[color=cyan] at (7.5, 3.5){32}; 
			
			\node[color=blue] at (.5, 2.5){10}; 
			\node[color=blue] at (1.5, 2.5){14}; 
			\node [color=blue] at (2.5, 2.5){19}; 
			\node [color=blue] at (3.5, 2.5){23}; 
			\node [color=blue] at (4.5, 2.5){27}; 
			\node at (5.5, 2.5){30}; 
			\node[color=cyan] at (6.5, 2.5){31}; 
			\node[color=cyan] at (7.5, 2.5){33}; 
			
			\node[color=blue] at (.5, 1.5){13}; 
			\node[color=blue] at (1.5, 1.5){18}; 
			\node[color=blue] at (2.5, 1.5){22}; 
			\node  [color=blue] at (3.5, 1.5){26}; 
			\node [color=blue] at (4.5, 1.5){29}; 
			\node [color=blue] at (5.5, 1.5){32}; 
			\node at (6.5, 1.5){34}; 
			\node[color=cyan] at (7.5, 1.5){35}; 
			
			\node[color=blue] at (.5, .5){17}; 
			\node[color=blue] at (1.5, .5){21}; 
			\node[color=blue] at (2.5, .5){25}; 
			\node  [color=blue] at (3.5, .5){28}; 
			\node [color=blue] at (4.5, .5){31}; 
			\node [color=blue] at (5.5, .5){33}; 
			\node [color=blue] at (6.5, .5){35}; 
			\node at (7.5, .5){36}; 
			
		\end{scope}
	\end{tikzpicture}
	\caption{ Filling of the rectangle when $r=3, 4, 5, 6, 7, 8$ giving $N^r_{g,g-1}$. }
	\label{fig_Nfilling}
\end{figure}

As the family of chains of elliptic curves has the expected dimension, the limit $\g{r}{g-1}$ smooths, (see~\cite[Theorem~2.1]{DBN}),  and we obtain a component of $\BN{g}{r}{g-1}$ 
whose general point is a curve admitting a $\g{r}{g-1}$. This $\g{r}{g-1}$ is not a theta-characteristic, as the filling is asymmetric.

\begin{defn}\label{defn_Ncomponent}
	We denote by $N^r_{g,g-1}$ the component of $\BN{g}{r}{d-1}$ obtained from the family of elliptic curves with a limit $\g{r}{g-1}$ and torsion conditions given by the fillings in \Cref{fig_Nfilling}.
\end{defn}

As a result of the torsion conditions imposed by a symmetric filling, we show that the components $ET^r_g$ and $N^r_{g,g-1}$ are distinct.

\begin{theorem}\label[theorem]{thredMgrd} For $r\geq 3$ and $g= {r+2\choose 2}$, the Brill--Noether locus ${\mathcal M}^r_{g,g-1}$ is reducible with one component consisting of curves with a theta-characteristic 
	and at least another component  with generic curve not possessing a theta-characteristic of dimension $r$.
\end{theorem}
\begin{proof}
From \Cref{exthch}, $\M^r_{g,g-1}$ has a component $ET^r_g$ of codimension ${r+1\choose 2}$ in ${\mathcal M}_g$ consisting of curves with a half-canonical $g_{g-1}^r$. We now show that $N^r_{g,g-1}\nsubseteq ET^r_g$, which also gives the reverse non-containment $ET^r_g \nsubseteq N^r_{g,g-1}$ as these component have the same dimension, thus providing two distinct components of $\BN{g}{r}{g-1}$.

We show in fact that the general curve of $N^r_{g,g-1}$ does not admit a theta-characteristic. Assume for contradiction that $N^r_{g,g-1}\subset ET^r_{g}$. Take a family of curves in $N^r_{g,g-1}$ degenerating to a chain of elliptic curves, $X$, with torsion specified by the fillings in \Cref{fig_Nfilling}. In particular, $X$ has only two components with a $2$-torsion condition. By assumption, $X$ has a limit theta-characteristic, giving a symmetric admissible filling of the $(r+1)\times(r+1)$ square with the entries $1,\dots, g$. As $r\geq 3$, the symmetric filling imposes at least one more $2$-torsion condition than the filling for $N^r_{g,g-1}$, and thus $X$ does not admit such a limit theta-characteristic, hence $N^r_{g,g-1}\nsubseteq ET^r_g$.

Thus, as $N^r_{g,g-1}$ and $ET^r_g$ were defined as the irreducible components of ${\mathcal M}^r_g$ degenerating to the respective chains,
 the general curve in $N^r_{g,g-1}$ does not admit a theta-characteristic $L$ with $h^0(L)=r+1$.
\end{proof}

\section{Brill--Noether loci with at least three components}

We show that for a range of $r$ and $g= {r+2 \choose 2}$, the Brill--Noether loci $\BN{g}{r}{g-1}$ have at least three components, the component $ET^r_g$ of $T^r_g$ found in \Cref{exthch} consisting of curves where the $\g{r}{g-1}$ is a theta-characteristic, a component $N^r_{g,g-1}$ where the $\g{r}{g-1}$ is not a theta-characteristic, and a component containing the locus $\BN{g}{1}{\frac{r+3}{2}}$. We note that the components $ET^r_g$ and $N^r_{g,g-1}$ have the expected dimension of $\BN{g}{r}{g-1}$. 

The component $N^r_{g,g-1}$ is constructed from the fillings in \Cref{fig_Nfilling}, see \Cref{defn_Ncomponent}. The admissible fillings giving $N^r_{g,g-1}$ for $r=3,4,5,6,7,8$ are shown in \Cref{fig_Nfilling}. 

\begin{remark}
	For $g={r+2 \choose 2}$, the $\g{r}{g-1}$ on curves in $N^r_{g,g-1}$ need not be very ample. When $r=3$ (and $g=10$), the limit $\g{3}{9}$ on the chain of elliptic curves given by the $N^3_{10,9}$ fillings is not very ample. Indeed, a limit $\g{2}{7}$ is obtained from the limit $\g{3}{9}$ by subtracting $2P_5$ (which amounts to subtracting $2Q_i$ for $i=1,2,3,4$ and subtracting $2P_i$ for $i=5,6,7,8,9,10$). The divisor $2P_5$ imposes one condition on the chain of elliptic curves, and subtracting $2P_5$ indeed gives a limit $\g{2}{7}$ corresponding to the filling in \Cref{fig:g27_filling}. While this does not show that the smoothed $\g{3}{9}$ is not very ample, it does show there is nothing forcing the $\g{3}{9}$ to be very ample. In fact, we'll observe that in this case the $\g{3}{9}$ is not very ample.
	\begin{figure}[h!]
		\begin{tikzpicture}[scale=.45]
			
			\begin{scope}[ ]
				\foreach \x in {0, 1,2,3,4,5} {\draw[thick] (0,\x) -- (3, \x); }
				\foreach \x in {0, 1,2,3} {\draw[thick] (\x,0) -- ( \x,5); }
				
				\node at (0.5, 4.5) {1};
				\node at (1.5, 4.5) {2};
				\node at (2.5, 4.5) {3};
				
				\node at (0.5, 3.5) {2};
				\node at (1.5, 3.5) {4};
				\node at (2.5, 3.5) {6};
				
				\node at (0.5, 2.5) {3};
				\node at (1.5, 2.5) {5};
				\node at (2.5, 2.5) {7};
				
				\node at (0.5, 1.5) {6};
				\node at (1.5, 1.5) {8};
				\node at (2.5, 1.5) {9};
				
				\node at (0.5, 0.5) {7};
				\node at (1.5, 0.5) {9};
				\node at (2.5, 0.5) {10};	
			\end{scope}
			
		\end{tikzpicture}
		\caption{  Admissible filling giving a $\g{2}{7}$. }
		\label{fig:g27_filling}
	\end{figure}
\end{remark}

\begin{remark}
	When $r=3$, and $g=10$, the existence of three components of $\BN{10}{3}{9}$ is in contrast with \cite[Exercise~1.41]{Harris_Morrison_book}, where the restricted Hilbert scheme $\mathcal{R}_{9,10,3}$ of degree $9$ curves of genus $10$ in $\mathbb{P}^3$ is shown to have two components: a component $\mathcal{J}_2$ where the $\g{3}{9}$ is not a theta-characteristic and the general curve is a divisor of type $(3,6)$ on a quadratic surface, and another component $\mathcal{J}_3$ where the $\g{3}{9}$ is a theta-characteristic and the general curve is a complete intersection of two cubic surfaces. 
	
	In particular, as we show that $\BN{10}{3}{9}$ has three components, the $\g{3}{9}$'s cannot be very ample for the general curve on all of the components. Matching the components of $\BN{10}{3}{9}$ and $\mathcal{R}_{9,10,3}$ we cannot find a match for the component $N^3_{10,9}$. Indeed, the general curve in $\mathcal{J}_2$ has gonality $3$ which should come from the component of $\BN{10}{3}{9}$ containing $\BN{10}{1}{3}$, and the theta-characteristic component $\mathcal{J}_3$ should come from $T^3_{10}$. Thus the $\g{3}{9}$ on the general curve in $N^3_{10,9}$ cannot be very ample, as otherwise we would obtain another component of the restricted Hilbert scheme.
\end{remark}

More precisely, we show that for $r=3,5,7$ or $r\geq 8$ and $g={r+2 \choose 2}$, the Brill--Noether locus $\BN{g}{r}{g-1}$ has at least three components, which are $ET^r_g$, $N^r_{g,g-1}$, and a component containing $\BN{g}{1}{\kappa(g,r,g-1)}$, where $\kappa(g,r,d)\coloneqq \max \left\{k ~\vert~ \BN{g}{1}{k}\subseteq \BN{g}{r}{d}\right\}$, as defined in \cite{AHL}.

\begin{remark} 
	For $r=4$ and $r=6$, we have $\kappa(g,r,g-1)=\floor{\frac{r+3}{2}}$, and \[\rho\left(g,1,\floor{\frac{r+3}{2}}\right)=\rho(g,r,g-1)-1,\] and we cannot exclude the possibility that $N^r_{g,g-1}$ or $ET^r_{g}$ contain the locus $\BN{g}{1}{\floor{\frac{r+3}{2}}}$.
\end{remark}

Since $ET^r_g$ and $N^r_{g,g-1}$ are distinct (see \Cref{thredMgrd}), it remains to show that the gonality locus gives another component. To this end, we investigate the gonality of the curves in the components $ET^r_g$ and $N^r_{g,g-1}$.

\begin{lemma}\label[lemma]{lemgon}
	Let $C$ be a chain of elliptic curves that is in the closure of the locus of $k$-gonal curves. 
	Then, in at least $-\rho$ of the elliptic components the points $P_i, Q_i$ satisfy $m_i(P_i-Q_i)=0$ and these $m_i$ necessarily satisfy  $m_i\le k$.
\end{lemma}
\begin{proof} If a chain of elliptic curves is the limit of $k$-gonal curves, then one can fill a $2\times (g-k+1)$ rectangle with the indices $1,2,\dots, g$.
	This requires a minimum of $2\times (g-k+1)-g$ repeats of indices. 
	Because one of the dimensions of the rectangle is 2, an index cannot appear more than twice.
	There are therefore at most $g-[2\times (g-k+1)-g]=2k-2$ non repeated indices in the filling of the rectangle.
	To obtain maximal torsion among the points $P_i, Q_i$, the two appearances of the index $i$ should be placed at the maximum possible distance.
	This can be achieved by placing the first $k-1$ indices in the first column to create a gap of $k$ then repeated indices one on each column and the last $k-1$ indices in the last column to close the gap.
	Note that if there are fewer non-repeated indices, the maximum $m_i$ will decrease rather than increase.
\end{proof}

\begin{lemma}\label[lemma]{lemgonT+N}
	The general curve $C\in N^r_{g,g-1}$ admits no $g^1_k$ for \[k\le \begin{cases}
		r, & \text{if $r$ is odd} \\
		r+1, & \text{if $r$ is even}
	\end{cases}.\]
	The general curve in $ET^r_g$ admits no $g^1_k$ for $k\le 2r-1$.
\end{lemma}
\begin{proof} In the filling of the square corresponding to a point in $N^r_{g,g-1}$, there are points $P_i, Q_i$ such that the torsion is $2\lceil{\frac {r+1}2}\rceil$. 
If the curve is $k$-gonal, \Cref{lemgon} shows $2\ceil{\frac{r+1}{2}}\leq k$.
 Thus the curve admits no $\g{1}{k}$ with 	\[k\leq \begin{cases}	r, & \text{if $r$ is odd} \\		r+1, & \text{if $r$ is even}	\end{cases}.\]
		
In the filling of the square corresponding to a point in $ET^r_{g}$, if $r\neq 3$ there are points $P_i, Q_i$ such that the torsion is $2r$.
 If the curve is $k$-gonal, \Cref{lemgon} again shows $2r\le k$. 
 Hence the curve admits no $\g{1}{k}$ with $k\leq 2r-1$.
  If $r=3$, it is well-known that the general curve of genus $10$ with a theta-characteristic is a complete intersection of two cubics, hence has gonality $6$, thus also not admitting a $\g{1}{k}$ with $k\leq 2r-1$.
\end{proof}

\begin{lemma}\label[lemma]{lem_kappa_rho_bounds}
	Let $r\geq 3$ be an integer, and $g={r+2 \choose 2}$. Then 
	\begin{itemize}
		\item if $r\geq 3$ we have $\kappa(g,r,g-1)\leq r$, and 
		\item if $r=8,10,12$ or $r\geq 14$, we have $\rho(g,1,\kappa(g,r,g-1))>\rho(g,r,g-1)$.
	\end{itemize}
\end{lemma}
\begin{proof}
	To prove the first assertion, we note that since $g+1 < \floor{\frac{g-1}{r}}+g-1$, \cite[Proposition~2.5]{AHL} gives
	 \[\kappa(g,r,g-1)=2r+2+\floor{-2\sqrt{-\rho(g,r,g-1)}}=2r+2 +\floor{-\sqrt{2}\sqrt{r(r+1)}}.\] 
	To verify that $2r+2 +\floor{-\sqrt{2}\sqrt{r(r+1)}}\leq r$ it suffices to show \[2r+2-\sqrt{2}\sqrt{r(r+1)}\leq r.\] 
	This can be seen to be equivalent to $r^2-2r-2\geq 0$, which holds for $r\geq 3$. 
	The first assertion follows.
	
	To prove the second assertion, we notice the equivalence
	 \[\rho(g,1,\kappa(g,r,g-1))>\rho(g,r,g-1) \iff 3r+1+2\floor{-\sqrt{2}\sqrt{r(r+1)}}>0,\] which follows from the computation 
	 \begin{eqnarray*}
	 & &\rho(g,1,\kappa) > \rho(g,r,g-1)\\
	 &\iff & g-2(g-\kappa+1) > g-(r+1)^2 \\
	 &\iff & (r+1)^2 > 2g-2\kappa +2 \\
	 &\iff & (r+1)^2 > 2g-2\left(2r+2+\floor{-\sqrt{2}\sqrt{r(r+1)}}\right) +2 \\
	 &\iff & r^2+6r+3+2\floor{-\sqrt{2}\sqrt{r(r+1)}} >2g\\
	 &\iff & r^2+6r+3+2\floor{-\sqrt{2}\sqrt{r(r+1)}} >(r+2)(r+1) \\
	 &\iff & 3r+1+2\floor{-\sqrt{2}\sqrt{r(r+1)}}>0  .
	 \end{eqnarray*} 
	This last inequality can be checked to hold for $r=8,10,12$ and $r\geq 14$.
\end{proof}

\begin{theorem}\label[theorem]{thm:3comp_BNloci}
	Let $r=3,5,7$ or $r\geq 8$ and $g={r+2\choose 2}$, then $\BN{g}{r}{g-1}$ has at least three components, $ET^r_g$, $N^r_{g,g-1}$, and a component containing $\BN{g}{1}{\kappa(g,r,g-1)}$.
\end{theorem}
\begin{proof}
	Let $\kappa=\kappa(g,r,g-1)$. For $r=3,5,7,9,11,13$, we note that 
	\[\dim N^r_{g,g-1}=3g-3+\rho(g,r,g-1)=\dim ET^r_g=3g-3-{r+1\choose 2}=\dim\BN{g}{1}{\kappa}= 3g-3+\rho\left(g,1,\kappa\right),\]
	 thus a non-containment in one direction gives the reverse non-containment as well. 
	 As \Cref{thredMgrd} shows that $N^r_{g,g-1}$ and $ET^r_{g}$ are distinct, it suffices to show that they are both distinct from $\BN{g}{1}{\kappa}$.
	  We note that for $r=3,5,7$, we have $\kappa(g,r,g-1)=\frac{r+3}{2}$, and \Cref{lemgonT+N} shows that $N^r_{g,g-1}\nsubseteq \BN{g}{1}{\frac{r+3}{2}}$ and that $ET^r_g\nsubseteq \BN{g}{1}{\frac{r+3}{2}}$, as desired. As $ET^r_g$, $N^r_{g,g-1}$ and $\BN{g}{1}{\kappa}$ have the same dimension, we obtain at least three distinct components for $r=3,5,7,9,11$, and $13$.
	
	 For $r=8,10,12$ or $r\geq 14$, we see that $\BN{g}{1}{\kappa}$ must give an additional component from the fact that
	 \[\dim \BN{g}{1}{\kappa}=3g-3+\rho(g,1,\kappa)>3g-3+\rho(g,r,g-1)=\dim N^r_{g,g-1}=\dim ET^r_g.\] Hence $\BN{g}{1}{\kappa}\nsubseteq N^r_{g,g-1}$ and $\BN{g}{1}{\kappa}\nsubseteq ET^r_{g}$ for dimension reasons, and as $\kappa(g,r,g-1)\leq r$, we see from \Cref{lemgonT+N} that $N^r_{g,g-1}\nsubseteq \BN{g}{1}{\kappa}$ and $ET^r_g \nsubseteq\BN{g}{1}{\kappa}$.
\end{proof}

\section{Irreducibility of Brill--Noether loci for $r=2$.}\label{sec:Irred BN loci}

In this section, we combine the irreducibility of the Severi variety with a bound on the dimension of covers of curves of smaller genus to obtain irreducibility of  Brill--Noether loci for $r=2$
for sufficiently large degree.
We then show that for smaller degree, the Brill--Noether loci can become reducible.

\begin{lemma}\label[lemma]{rho bound on component coming from covering}
	Suppose there is a component $\Xi$ of $\BN{g}{r}{d}$ such that for the generic  curve $C\in\Xi$, the map given by the $\g{r}{d}$, $\varphi:C\to D\subset \PP^r$, factors through a curve $D$ of geometric genus $\gamma\ge 1$.
	Let $k\geq 2$ be the degree of the map $\varphi:C\to D$. Then $3g-3+\rho(g,r,d)\leq \dim \Xi \leq 2g-2+(3-2k)(\gamma-1)$.
	In particular, if $\rho(g,r,d)>-g+1$ and for the generic curve $C\in\Xi$, the map given by the $\g{r}{d}$ is not birational, then the geometric genus of $D$ is 0. 
\end{lemma}
\begin{proof} From \cite{Lange}, the dimension of the moduli space of curves that have a degree $k$ map to a curve of genus $\gamma\ge 1$ is 
	$\dim M(\gamma ,n) = 2g-2-(2k-3)(\gamma-1)$.
	From the local definition of the Brill--Noether locus as a determinantal variety, $3g-3+\rho(g,r,d)\leq \dim \Xi$.
	Therefore
	\[3g-3+\rho(g,r,d)\leq \dim \Xi \leq  2g-2-(2k-3)(\gamma-1).\]
	If $\rho(g,r,d)> - g+1$, the equation above gives $2g-2<   2g-2-(2k-3)(\gamma-1)$ which implies $\gamma=0$.
\end{proof}

\begin{prop}\label[proposition]{irredM2gd}
	If $\rho(g,2,d)<0$ and either
	\begin{enumerate}
		\item $g\geq 3$ and $d>\frac{g+4}{2}$, or
		\item $g\geq 4$, $d$ is odd, and $d=\frac{g+4}{2}$,
	\end{enumerate} 
	then $\BN{g}{2}{d}$ has a unique irreducible component containing a curve with a basepoint free $\g{2}{d}$, and any other component is comprised of curves where the $\g{2}{d}$ has basepoints.
\end{prop}

\begin{proof}
	The Severi variety of plane curves of degree $d$ and genus $g$ is irreducible (see \cite{Harris_Severi_86}) and has general element a curve containing at worst nodes as singularities, and gives a component of $\BN{g}{2}{d}$ where the general curve has a basepoint free $\g{2}{d}$. Any other component admitting curves with a basepoint free $\g{2}{d}$ would correspond to a curves with a non-birational map to the projective plane.
	Suppose then that $\Xi$ is such a component of $\BN{g}{2}{d}$.
	Let $C$ be a generic point of $\Xi$  and  $\varphi:C\to D\subset \PP^2$ the map (of degree $k\ge 2$) associated to the $\g{2}{d}$. 
	
	In case (1) and (2), one can readily check that $\rho(g,2,d)>-g+1$.
	By \Cref{rho bound on component coming from covering}, we must have the genus $\gamma $ of $D$ satisfies $\gamma=0$.
	
	As the map $\varphi$ associated to the $\g{2}{d}$ is nondegenerate, $D$ is birationally embedded in $\PP^2$ via a $\g{2}{e}$ on a smooth model of $D$
	and therefore	 $e\geq 2$, thus $2k\leq ek=d$. Now \Cref{rho bound on component coming from covering} gives 
	$g+2+\rho(g,2,d)\leq d$, which is equivalent to $d\leq \frac{g+4}{2}$. Thus in case (1) $\BN{g}{2}{d}$ has a unique component coming from the image of the Severi variety.
	
	Now suppose we are in case (2). 
	As the genus of $D$ is $0$, we have $3g-3+\rho(g,2,d)\leq 2g-5+2k$, which is equivalent to $3d\leq g+4+2k$. As $2k\leq d$ and by assumption $2d= g+4$, then \[3d\leq g+4+2k\leq 2d+d,\] 
	thus in fact $2k=d$, contradicting the  assumption that $d$ is odd. 
	
	Thus $\BN{g}{2}{d}$ has a unique component where the general curve has a basepoint free $\g{2}{d}$ coming from the image of the Severi variety.
\end{proof}

\begin{remark}
	We note that when $\rho(g,2,d)=-2$, the fact that the curves in every component of $\BN{g}{2}{d}$ have a basepoint free $\g{2}{d}$ follows from codimension bounds for components of Brill--Noether loci, and thus $\BN{g}{2}{d}$ is irreducible in this case, as observed in \cite{CHOI2022}. In future work, we plan to investigate the range where the basepoint free assumption holds.
\end{remark}

We show that the bound in \Cref{irredM2gd} is optimal with three examples, with at least two components namely  $\BN{11}{2}{7}$,  $\BN{12}{2}{8}$ and $\BN{12}{2}{7}$.
Note that in the first example, $7=\frac {11+3}2$ is right below the bound. 
In the second example  the degree $8=\frac{12+4}2$ is  right at the bound but even.

\begin{remark}\label[remark]{rmk_multiplesofg1k}
	We note that for an arbitrary integer $2\leq k \leq \frac{d}{r}$, we have $\BN{g}{1}{k}\subseteq \BN{g}{r}{d}$, since $|r\g{1}{k}+(d-rk)\text{points}|$ gives a $\g{s}{d}$ for some $s\geq r$.
\end{remark}

\begin{example}
We show that $\BN{11}{2}{7}$ has at least two components. 
The image of the Severi variety gives a component $\Xi$ of $\BN{11}{2}{7}$ whose general element admits a $\g{2}{7}$ that gives a map $\phi:C\to \PP^2$ whose image has $4$ nodes as singularities, 
and \cite[Theorem 2.1]{Coppens_Kato} shows that these curves have no $\g{1}{4}$, hence have gonality $\geq 5$ and indeed a $\g{1}{5}$ is given by projecting from a node.
 On the other hand, 
 from \Cref{rmk_multiplesofg1k}, we see that $\BN{11}{1}{3}\subseteq \BN{11}{2}{7}$, hence $\BN{11}{1}{3}$ is contained in a component of $\BN{11}{2}{7}$. 
 As $\dim(\BN{11}{1}{3})=\dim(\Xi)$, and the general curve of $\Xi$ admits no $\g{1}{3}$, the component containing $\BN{11}{1}{3}$ cannot be $\Xi$, and so $\BN{11}{2}{7}$ has at least two components.
\end{example}

\begin{example}
Similarly, we show that $\BN{12}{2}{8}$ has at least two components. 
The image of the Severi variety gives a component whose general element admits a $\g{2}{8}$ that gives a map $\phi:C\to\PP^2$ whose image has $9$ nodes as singularities, 
and \cite[Theorem 2.1]{Coppens_Kato} shows that these curves have gonality $6$. 
The other component comes from curves where the map corresponding to the $\g{2}{8}$ factors as $C\to D\to \PP^2$, where $D$ has geometric genus $0$ and $D\to\PP^2$ is given by a complete $\g{2}{2}$ on $D$. 
One can readily check that the dimension count in \cite{Lange} shows that this is the only possibility for the geometric genus of $D$. 
A curve $C\in\BN{12}{2}{8}$ of gonality $4$ gives exactly such a curve, and 
Remark \ref{rmk_multiplesofg1k} shows that $\BN{12}{1}{4}\subset \BN{12}{2}{8}$, giving the second component of $\BN{12}{2}{8}$ whose general element has gonality $4$. 
\end{example}

The construction of \Cref{example_12_2_7} of Brill--Noether loci with multiple components is quite general, as observed in \cite[Appendix A.1]{pflueger_lego}.

\begin{remark}[{\cite[Appendix~A.1]{pflueger_lego}}]\label[remark]{thm_large_gon_comp}
	Suppose that $0>\rho(g,1,k)>\rho(g,r,d)$, $\BN{g}{r}{d}$ has a component of expected dimension, and $\BN{g}{1}{k}\subseteq \BN{g}{r}{d}$. Then $\BN{g}{r}{d}$ has at least two components. Indeed, as $\dim \BN{g}{1}{k} = 3g-3 +\rho(g,1,k)> 3g-3 + \rho(g,r,d)= \operatorname{exp}\dim \BN{g}{r}{d}$, the component of $\BN{g}{r}{d}$ of expected dimension cannot contain $\BN{g}{1}{k}$.
\end{remark}

When $r=2$, this gives Brill--Noether loci $\BN{g}{2}{d}$ that always have at least two components for some range of $d$, one component of expected dimension, and another coming from a non-trivial containment of Brill--Noether loci $\BN{g}{1}{k}\subseteq \BN{g}{2}{d}$.

\begin{cor}
	Suppose $g>8$ and $\floor{\frac{g+4}{2}}>d\geq \frac{g}{3}+3$, then the Brill--Noether locus $\BN{g}{2}{d}$ has at least two components.
\end{cor}
\begin{proof}
We note that the assumption $d\geq \frac{g}{3}+3$ is equivalent to $\rho(g,2,d)\geq -g+3$, hence \cite[Theorem A]{pflueger_lego} or \cite[Theorem 2.1]{DBN}
 give a component of $\BN{g}{2}{d}$ of the expected dimension $3g-3+\rho(g,2,d)$, giving one of the claimed components.

On the other hand  $\BN{g}{1}{k}\subset \BN{g}{2}{d}$ for $k=\floor{\frac{d}{2}}$. 
If $3g-3+\rho(g,1,k)=\dim(\BN{g}{1}{k})>3g-3+\rho(g,2,d)$, then $\BN{g}{2}{d}$ must have a second component of larger dimension which contains $\BN{g}{1}{k}$.
It remains to show that $3g-3+\rho(g,1,k)>3g-3+\rho(g,2,d)$, or equivalently that 
	\begin{equation}
		\rho(g,1,k)>\rho(g,2,d). \label{eq:mg2d_red_ineq}
	\end{equation} 
Expanding and recalling that $k=\floor{\frac{d}{2}}$, we see that \eqref{eq:mg2d_red_ineq} follows from the assumption $d<\floor{\frac{g+4}2}$.
\end{proof}

\begin{example}\label[example]{example_12_2_7} From the corollary above, we see  that $\BN{12}{2}{7}$ has at least two components as 
$\floor{\frac{g+4}{2}}=\frac{12+4}2=8>7=d=\frac{12}3+3= \frac{g}{3}+3$.
\end{example}

\vspace*{1.5cm}


\begin{thebibliography}{10}
	
	\bibitem{AHL}
	Auel, A., Haburcak, R., Larson, H.:
	\newblock Maximal {B}rill--{N}oether loci via the gonality stratification.
	\newblock {Mich. Math. J., to appear}, 2024.
	
	\bibitem{Balico_fontanari}
	Ballico, E., Fontanari, C.:
	\newblock A few remarks about the {H}ilbert scheme of smooth projective curves.
	\newblock {Comm. Algebra}, \textbf{42}(9), 3895--3901, 2014.
	
	\bibitem{Benzo}
	Benzo, L.:
	\newblock Components of moduli spaces of spin curves with the expected
	codimension.
	\newblock {Math. Ann.}, \textbf{363}(1-2), 385--392, 2015.
	
	\bibitem{CIK1}
	Choi, Y., Iliev, H., Kim, S.:
	\newblock Reducibility of the {H}ilbert scheme of smooth curves and families of
	double covers.
	\newblock {Taiwanese J. Math.}, \textbf{21}(3), 583--600, 2017.
	
	\bibitem{CIK2}
	Choi, Y., Iliev, H., Kim, S.:
	\newblock Components of the {H}ilbert scheme of smooth projective curves using
	ruled surfaces.
	\newblock {Manuscripta Math.}, \textbf{164}(3-4), 395--408, 2021.
	
	\bibitem{CIK3}
	Choi, Y., Iliev, H., Kim, S.:
	\newblock Non-reduced components of the {H}ilbert scheme of curves using triple
	covers.
	\newblock {Mediterr. J. Math.}, \textbf{21}(4), Paper No. 134, 20, 2024.
	
	\bibitem{CHOI2022}
	Choi, Y., Kim, S.:
	\newblock Linear series on a curve of compact type bridged by a chain of
	elliptic curves.
	\newblock {Indag. Math. (N.S.)}, \textbf{33}(4), 844--860, 2022.
	
	\bibitem{Coppens_Kato}
	Coppens, M., Kato, T.:
	\newblock The gonality of smooth curves with plane models.
	\newblock {Manuscripta Math.}, \textbf{70}(1), 5--25, 1990.
	
	\bibitem{EinHilbP^3}
	Ein, L.:
	\newblock Hilbert scheme of smooth space curves.
	\newblock {Ann. Sci. \'Ecole Norm. Sup. (4)}, \textbf{19}(4), 469--478, 1986.
	
	\bibitem{EinHilbP^r}
	Ein, L.:
	\newblock The irreducibility of the {H}ilbert scheme of smooth space curves.
	\newblock In: {Algebraic geometry, {B}owdoin, 1985 ({B}runswick, {M}aine,
		1985)}, volume 46, Part 1 of {Proc. Sympos. Pure Math.}, pages 83--87.
	Amer. Math. Soc., Providence, RI, 1987.
	
	\bibitem{EH}
	Eisenbud, D., Harris, J.:
	\newblock Irreducibility of some families of linear series with
	{B}rill-{N}oether number {$-1$}.
	\newblock {Ann. Sci. \'Ecole Norm. Sup. (4)}, \textbf{22}(1), 33--53, 1989.
	
	\bibitem{F}
	Farkas, G.:
	\newblock Gaussian maps, {G}ieseker-{P}etri loci and large
	theta-characteristics.
	\newblock {J. Reine Angew. Math.}, \textbf{581}, 151--173, 2005.
	
	\bibitem{HHilbert}
	Harris, J.:
	\newblock {Curves in projective space}, volume~85 of {S\'eminaire de
		Math\'ematiques Sup\'erieures [Seminar on Higher Mathematics]}.
	\newblock Presses de l'Universit\'e{} de Montr\'eal, Montreal, QC, 1982.
	\newblock With the collaboration of David Eisenbud.
	
	\bibitem{HTheta}
	Harris, J.:
	\newblock Theta-characteristics on algebraic curves.
	\newblock {Trans. Amer. Math. Soc.}, \textbf{271}(2), 611--638, 1982.
	
	\bibitem{Harris_Severi_86}
	Harris, J.:
	\newblock On the {S}everi problem.
	\newblock {Invent. Math.}, \textbf{84}(3), 445--461, 1986.
	
	\bibitem{Harris_Morrison_book}
	Harris, J., Morrison, I.:
	\newblock {Moduli of curves}, volume 187 of {Graduate Texts in
		Mathematics}.
	\newblock Springer-Verlag, New York, 1998.
	
	\bibitem{I}
	Iliev,H.:
	\newblock On the irreducibility of the {H}ilbert scheme of curves in {$\Bbb
		P^5$}.
	\newblock {Comm. Algebra}, \textbf{36}(4), 1550--1564, 2008.
	
	
	\bibitem{KK}
	Keem, C., Kim, S.:
	\newblock Irreducibility of a subscheme of the {H}ilbert scheme of complex
	space curves.
	\newblock {J. Algebra}, \textbf{145}(1), 240--248, 1992.
	
	\bibitem{KKi}
	Keem, C., Kim, Y.-H.:
	\newblock Irreducibility of the {H}ilbert scheme of smooth curves in
	{$\Bbb{P}^3$} of degree {$g$} and genus {$g$}.
	\newblock {Arch. Math. (Basel)}, \textbf{108}(6), 593--600, 2017.
	
	\bibitem{Lange}
	Lange, H.:
	\newblock Moduli spaces of algebraic curves with rational maps.
	\newblock {Math. Proc. Cambridge Philos. Soc.}, \textbf{78}(2), 283--292, 1975.
	
	\bibitem{LVV}
	Larson, E., Vakil, R., Vogt,I.:
	\newblock The interpolation problem: when can you pass a curve of a given type
	through {$n$} random points in space?
	\newblock {Bull. Amer. Math. Soc. (N.S.)}, \textbf{62}(1), 67--91, 2025.
	
	\bibitem{LOTZ1}
	Liu, F., Osserman, B., Teixidor~i Bigas, M., Zhang, N.:
	\newblock Limit linear series and ranks of multiplication maps.
	\newblock {Trans. Amer. Math. Soc.}, \textbf{374}(1), 367--405, 2021.
	
	\bibitem{LOTZ2}
	Liu, F., Osserman, B., Teixidor~i Bigas, M., Zhang, N.:
	\newblock The strong maximal rank conjecture and moduli spaces of curves.
	\newblock {Algebra Number Theory}, \textbf{18}(8), 1403--1464, 2024.
	
	\bibitem{LT}
	L\'opez~Mart\'in, A., Teixidor~i Bigas,M.:
	\newblock Limit linear series on chains of elliptic curves and tropical
	divisors on chains of loops.
	\newblock {Doc. Math.}, \textbf{22}, 263--286, 2017.
	
	\bibitem{MMurphyL}
	Mumford, D.:
	\newblock Further pathologies in algebraic geometry.
	\newblock {Amer. J. Math.}, \textbf{84}, 642--648, 1962.
	
	\bibitem{MTheta}
	Mumford, D.:
	\newblock Theta characteristics of an algebraic curve.
	\newblock {Ann. Sci. \'Ecole Norm. Sup. (4)}, \textbf{4}, 181--192, 1971.
	
	\bibitem{pflueger_lego}
	Pflueger, N.:
	\newblock Linear series with {$\rho<0$} via thrifty {L}ego building.
	\newblock {J. Reine Angew. Math.}, \textbf{797}, 193--228, 2023.
	
	\bibitem{semican}
	Teixidor~i Bigas, M.:
	\newblock Half-canonical series on algebraic curves.
	\newblock {Trans. Amer. Math. Soc.}, \textbf{302}(1), 99--115, 1987.
	
	\bibitem{thetdiv}
	Teixidor~i Bigas, M.:
	\newblock The divisor of curves with a vanishing theta-null.
	\newblock {Compositio Math.}, \textbf{66}(1), 15--22, 1988.
	
	\bibitem{Ramif}
	Teixidor~i Bigas, M.:
	\newblock Brill-{N}oether loci with ramification at two points.
	\newblock {Ann. Mat. Pura Appl. (4)}, \textbf{202}(3), 1217--1232, 2023.
	
	\bibitem{DBN}
	Teixidor~i Bigas, M.:
	\newblock Brill-{N}oether loci.
	\newblock {Manuscripta Math.}, \textbf{176}(1), Paper No. 14, 17, 2025.
	
	\bibitem{VMurphyL}
	Vakil, R.:
	\newblock Murphy's law in algebraic geometry: badly-behaved deformation spaces.
	\newblock {Invent. Math.}, \textbf{164}(3), 569--590, 2006.
	
\end{thebibliography}

\vfill

\end{document}